\documentclass[10pt,a4paper]{article}

\usepackage{amsmath}
\usepackage{amssymb}
\usepackage{empheq}
\usepackage{stackrel}
\usepackage{amsthm}

\newcommand{\ds}{\displaystyle}

\newcommand{\lt}{\left}
\newcommand{\rt}{\right}
\def\CC{{\mathbb C}}
\def\RR{{\mathbb R}}
\def\ZZ{{\mathbb Z}}
\def\NN{{\mathbb N}}
\providecommand{\arccosh}{\mathop{\rm arccosh}\nolimits}
\providecommand{\Vol}{\mathop{\rm Vol}\nolimits}
\providecommand{\CS}{\mathop{\rm CS}\nolimits}

\providecommand{\im}{\mathop{\rm Im}\nolimits}
\def \Re{{\operatorname{Re}}}
\newtheorem{conjecture}{Conjecture}
\newtheorem{theorem}{Theorem}
\newtheorem{lemma}{Lemma}
\newtheorem{proposition}{Proposition}
\newtheorem{remark}{Remark}

\title{Asymptotic Behavior of the Colored Jones polynomials and Turaev-Viro Invariants of the figure eight knot}

\author{Ka Ho WONG\footnote{The first author is partially supported by the Mathematics department of the Chinese University of Hong Kong.} , Thomas Kwok-Keung AU}

\date{ }

\begin{document}

\maketitle

\begin{abstract}
In this paper we investigate the asymptotic behavior of the colored Jones polynomials and the Turaev-Viro invariants for the figure eight knot. More precisely, we consider the $M$-th colored Jones polynomials evaluated at $(N+1/2)$-th root of unity with a fixed limiting ratio, $s$, of $M$ and $(N+1/2)$. We find out the asymptotic expansion formula (AEF) of the colored Jones polynomials of the figure eight knot with $s$ close to $1$. Nonetheless, we show that the exponential growth rate of the colored Jones polynomials of the figure eight knot with $s$ close to $1/2$ is strictly less than those with $s$ close to $1$. It is known that the Turaev Viro invariant of the figure eight knot can be expressed in terms of a sum of its colored Jones polynomials. Our results show that this sum is asymptotically equal to the sum of the terms with $s$ close to 1. As an application of the asymptotic behavior of the colored Jones polynomials, we obtain the asymptotic expansion formula for the Turaev-Viro invariants of the figure eight knot. Finally, we suggest a possible generalization of our approach so as to relate the AEF for the colored Jones polynomials and the AEF for the Turaev-Viro invariants for general hyperbolic knots.
\end{abstract}

\section{Introduction}

This paper aims to find out the asymptotic expansion formula (AEF) for the $M$-th colored Jones polynomials of the figure eight knot at $(M+a)$-th root of unity, with $a$ and $M$ satisfying some limiting relation. The method is motivated by the work in \cite{WA17} in which an asymptotic expansion of an $SU(n)$-invariant of the figure eight knot is given. In particular, we are interested in the case where $a = (N-M+1/2)$ with $N>M$, where $M \in \NN$ is a sequence of integers in $N$ with limiting ratio $\ds s=\lim_{N \to \infty}\frac{M}{N+1/2}$ close to $1/2$ or $1$. From the AEF of the colored Jones polynomials of the figure eight knot, we find out explicitly the large $r$ behavior of the Turaev-Viro invariants $TV_r(\mathbb{S}^3 \backslash 4_1)$. This tells us what kinds of topological information can be extracted from the AEF of the TV invariant.

\subsection{Overview of the volume conjecture}

The main theme of this paper is to establish the AEF for the Turaev-Viro invariants of the figure eight knot complement. The study of the volume conjecture of the Turaev-Viro invariants started from \cite{CY15}, in which Q.Chen and T.Yang discovered a version of volume conjecture of Turaev-Viro invariants at a $2r$-th root of unity where $r$ is an odd integer. The conjecture can be stated as follows.

\begin{conjecture}\label{cjtv}
For every hyperbolic 3-manifold M, we have
\[\lim_{r\to \infty}\frac{2\pi}{r}\log \left(TV_{r}(M,e^{\frac{2\pi i}{r}})\right) = \Vol(M)\]
where $r$ is an odd positive integer.
\end{conjecture}

This result is surprising since according to the Witten's Asymptotic Expansion conjecture, the Reshetikhin-Turaev invariant and Turaev-Viro invariants at a $4r$-th root of unity should grow polynomially in $r$. Furthermore, in their article, Chen and Yang found numerical evidence of Conjecture 1 for the figure eight knot complement, as well as other 3-manifolds. Besides, numerical evidenece shows that $\frac{2\pi}{r}\log \left(TV_{r}(M,e^{\frac{2\pi i}{r}})\right)$ goes faster to the hyperbolic volume than $\frac{2\pi}{r}\log\left|J_{r}(K;e^{\frac{2\pi i}{r}})\right|$.

To explain the gap between the asymptotic behaviors of these invariants, we need to relate the Turaev-Viro invariants with the colored Jones polynomials so that a comparison can be done. Recall that for a link $L$, the normalized colored Jones polynomials $J_M(L;q)$ and the unnormalized colored Jones polynomials $\tilde{J}_M(L;q)$ are related by
$$ \tilde{J}_M(L;q) = \frac{q^{M/2}-q^{-M/2}}{q^{1/2}-q^{-1/2}} \times J_M(L;q)$$
The following relationship between the colored Jones polynomials and the Turaev-Viro invariants is given by Theorem 1.1 in \cite{DKY17}.
\begin{theorem}\label{relationship}
Let L be a link in $\mathbb{S}^3$ with n components. Then given an odd integer $r=2N+1 \geq 3$, we have
\[ TV_{r}\left(\mathbb{S}^3 \backslash L, e^{\frac{2\pi i}{r}}\right) = 2^{n-1}(\eta_{r}')^{2} \sum_{1\leq M \leq \frac{r-1}{2}}\left|\tilde{J}_{M}\left(L,e^{\frac{2\pi i }{N+\frac{1}{2}}}\right)\right|^2 ,\]
where $\tilde{J}_M$ is the unnormalized colored Jones polynomials for the link $L$ and $$\eta_{r}' = \frac{2\sin(\frac{2\pi}{r})}{\sqrt{r}}$$
\end{theorem}

From this relationship between the Turaev-Viro invariants and the colored Jones polynomials, Conjecture~\ref{cjtv} has been proved for the case of the figure eight knot complement (Theorem 1.6 in \cite{DKY17}). Furthermore, in order to find out the AEF of the Turaev-Viro invariants, it is natural to consider the AEF of the $M$-th colored Jones polynomials, where $M=1,2,\dots N$.

The asymptotic behavior of the colored Jones polynomials has been investigated for a very long time. The classical volume conjecture (Conjecture~\ref{cvc} below) states that the evaluation of $N$-th colored Jones polynomials of a knot $K$ at an $N$-th root of unity captures the simplicial volume of the knot complement $\mathbb{S}^{3} \backslash K$.

\begin{conjecture}\label{cvc}(Classical volume conjecture \cite{K97,MM01}) Let $K$ be a knot and $J_{N}(K;q)$ be the $N$-th colored Jones polynomials of $K$ evaluated at $q$. We have
\[\lim_{N \to \infty} \frac{\log|J_{N}(K;e^{\frac{2\pi i}{N}})|}{N} = \frac{\operatorname{Vol}(\mathbb{S}^{3} \backslash K)}{2\pi}, \]
where $\operatorname{Vol}(\mathbb{S}^{3} \backslash K)$ is the simplicial volume of the knot complement.
\end{conjecture}

In \cite{AH06} Andersen and Hansen used saddle point approximation to find out the AEF for the $N$-th colored Jones polynomials of the figure eight knot evaluated at $N$-th root of unity.

\begin{theorem}\label{cvc2}
The AEF for the $N$-th colored Jones polynomials of the figure eight knot evaluated at $N$-th root of unity is given by
\begin{align*}
J_{N} \left(4_1; e^{\frac{2\pi i}{N}} \right)
&\stackrel[N \to \infty]{\sim}{ } \frac{1}{3^{1/4}}N^{3/2} \exp \left( \frac{N \Vol(\mathbb{S}^3 \backslash 4_1)}{2\pi} \right) \\
&\hspace{8pt} = \hspace{8pt} - 2\pi^{3/2} \left( \frac{2}{\sqrt{-3}} \right)^{1/2} \left( \frac{N}{2\pi i} \right)^{3/2} \exp \left( \frac{N}{2\pi i} \times i \Vol (\mathbb{S}^3 \backslash 4_1) \right).
\end{align*}
\end{theorem}

As a generalization of Theorem~\ref{cvc2}, in \cite{HM13} H. Murakami refined the asymptotic expansion formula of the colored Jones polynomials, capturing the Chern-Simons invariant together with the Reidemeister torsion of the knot. (A related result on colored HOMFLY polynomial is obtained in \cite{WA17}.) 

\begin{theorem}\label{asymsu2}
{\em (Asymptotic expansion formula for the colored Jones polynomials of the figure eight knot \cite{HM13})\/}
Let $u$ be a real number with $0 < u < \log((3+\sqrt{5})/2)$ and put $\xi = 2\pi i + u$. Then we have the following asymptotic equivalence for the colored Jones polynomials of the figure-eight knot~$4_{1}$:
\begin{align*}
J_{N} \lt( 4_{1} ; e^{\frac{\xi}{N}} \rt) \stackrel[N \to \infty]{\sim}{ }  \frac{\sqrt{-\pi}}{2 \sinh (u/2)} T(u)^{1/2} \left( \frac{N}{\xi}\right)^{1/2} \exp \left(\frac{N}{\xi} S(u)\right),
\end{align*}
where \[ S(u) = \operatorname{Li}_{2}\left(e^{u-\varphi(u)}\right) - \operatorname{Li}_{2}\left(e^{u+\varphi(u)}\right) - u\varphi(u) \]
and
\[ T(u) = \frac{2}{\sqrt{(e^{u} + e^{-u} +1)(e^{u} + e^{-u} -3)}}.\]
Here $\ds \varphi(u) = \arccosh(\cosh(u) - 1/2)$ and
\[ \operatorname{Li}_{2}(z) = - \int_{0}^{z} \frac{ \log(1-x)}{x} dx\] is the dilogarithm function.
\end{theorem}

In particular, the functions $S(u)$ and $T(u)$ are the Chern-Simons invariant and the cohomological twisted Reidemeister torsion respectively, both of which are associated with an irreducible representation of $\pi_1 (\mathbb{S}^3 \backslash 4_1)$ into $SL(2;\mathbb{C})$ sending the meridian to an element with eigenvalues $\exp(u/2)$ and $\exp(-u/2)$ \cite{HM13}.

\subsection{AEF of the $M$-th colored Jones polynomials of the figure eight knot at $(N+\frac{1}{2})$-th root of unity}\text{ }\\
The main ingredient of the proof of Theorem 1.6 in \cite{DKY17} is to find out an upper bound for the colored Jones polynomials of the figure eight knot. Recall that the $M$-th colored Jones polynomials of the figure eight knot is given by
$$ J_{M}(4_{1};q) =  1+\sum_{k=1}^{M-1} \prod_{l=1}^{k}\left(q^{\frac{M-l}{2}} - q^{-\frac{M-l}{2}} \right)\left(q^{\frac{M+l}{2}} - q^{-\frac{M+l}{2}} \right)$$

For $k,l \in  \{1,2,\dots, M-1\}$, let $\ds h_M(l) = \left(q^{\frac{M-l}{2}} - q^{-\frac{M-l}{2}} \right)\left(q^{\frac{M+l}{2}} - q^{-\frac{M+l}{2}} \right)$ and $g_{M}(k) = \prod_{l=1}^{k} \left| h_M (l)\right|$. The estimation for the upper bound of $g_M(k)$ can be stated as follows.

\begin{lemma}\label{uppbd1}
 For each $M$, let $k_{M} \in \{ 1,\dots, M-1 \}$ such that $g_M(k_M)$ achieves the maximum among all $g_M (k)$. Assume that $ \frac{M}{r}=\frac{M}{2N+1} \to d \in [0,\frac{1}{2}]$ and $\frac{k_M}{r} \to k_{d}$ as $r \to \infty$. Then we have
\[ \lim_{r \to \infty}\frac{1}{r} \log( g_M (k_M) ) = -\frac{1}{2\pi}\left( \Lambda(2\pi(k_d -d)) + \Lambda(2\pi(k_d +d)) \right) \leq \frac{\Vol(\mathbb{S}^3 \backslash 4_1)}{4\pi}.\]
Furthermore, the equality holds if and only if $(s=2d=1$ and $2k_d = \frac{5}{6})$ or $(s=2d=\frac{1}{2}$ and $2k_d = \frac{1}{3})$.
\end{lemma}
Lemma~\ref{uppbd1} follows easily from the argument in the proof of Theorem 1.6 in \cite{DKY17}. From this lemma we can see that for the figure eight knot, in order to find out the dominant terms among all the colored Jones polynomials in Theorem~\ref{relationship}, we only need to consider those $M$ satisfying $d=\frac{1}{2}$ or $d=\frac{1}{4}$.

This lemma motivates us to consider the AEF of the colored Jones polynomials of the figure eight knot with $s = 2d \sim 1$ or $s = 2d \sim \frac{1}{2}$. The key arguments of our proof is a one parameter family version of saddle point approximation, which allows us to generalize Theorem~\ref{asymsu2} to the case where $s \sim 1$. This technique has already been used to find out an analogue result of Theorem~\ref{asymsu2} for the colored HOMFLY polynomial of the figure eight knot (see \cite{WA17} for more details).

\begin{theorem}\label{FSA}{\em (One-parameter family version for the saddle point approximation)}
Let $\{ \Phi_{y}(z)\}_{y \in [0,1]}$ be a family of holomorphic functions smoothly depending on $y \in [0,1]$. Let $C(y,t): [0,1]^{2} \to \CC$ be a continuous family of contours with length uniformly bounded above by a fixed constant $L$, such that for each $y \in [0,1]$, $C(y,t)$ lies inside the domain of $\Phi_{y}(z)$, for which $z_{y}$ is the only saddle point along the contour $C(y,t)$ and $\max\operatorname{Re}\left[\Phi_{y}(z)\right]$ is attained at $z_{y}$. Further assume that $\left|\arg\left( \sqrt{-\frac{d^{2}\Phi_{0}}{dz^{2}}(z_{0})}\right)\right| < \pi /4$. Suppose we have an analytic function f(z) along the contour such that $f(z_{0}) \neq 0$. Then  for any sequence $\{y_{M} \}_{M \in \mathbb{N}}$ with $y_{M} \to 0$ as $M \to \infty$, we have the following generalized saddle point approximation:
\begin{align*}
&\int_{C(y_{M},t)} f(z) \exp (M \Phi_{y_{M}}(z)) dz \\
=  &\sqrt{\frac{2\pi}{M\left(-\frac{d^{2} \Phi_{y_{M}}}{dz^{2}} (z_{y_{M}})\right)}}\, f(z_{y_M})\exp(M \Phi_{y_{M}}(z_{y_{M}})) \left(1 + O\left(\frac{1}{M}\right)\right).
\end{align*}
\end{theorem}

\subsection{Main Results}

The main results of this paper are summarized as follows. First of all we consider the $M$-th\linebreak colored Jones polynomials around $(M+a)$-th root of unity $q = \exp\left(\frac{2\pi i + u}{M+a}\right)$ with some fixed non-negative real number $a$. One can easily see that when $a \in \NN$ and $u = 0$, we have $\ds \lim_{M \to \infty} J_{M}(4_{1},q) = 1$. So in the following discussion we only focus on the other situations.

\begin{theorem}\label{mainthm1} For $q = \exp\left(\frac{2\pi i + u}{M+a}\right)$, if $a \notin \NN$ or $u \neq 0$, we have
\begin{align}
J_{M}\left(4_1; q \right)
\stackrel[M \to \infty]{\sim}{ }  & \exp(a\varphi(u))\frac{(1-e^{u-\varphi(u)})^a}{(1 - e^{u+\varphi(u)})^a} \notag \\
&\quad \times  \frac{ \sqrt{-\pi} }{2\sinh (u/2)} T(u)^{1/2}\left(\frac{M+a}{\xi}\right)^{1/2} \exp\left((M+a)(S(u))\right)
\end{align}where
$S(u)$, $T(u)$ and $\varphi(u)$ are the functions appearing in Theorem~\ref{asymsu2}.
\end{theorem}

In particular, for the case where $u=0$, we have

\begin{theorem}\label{mainthm1.5}
When $u=0$ and $a\notin \NN$, we have
\begin{align}
J_{M}\left(4_1; q \right)
\stackrel[M \to \infty]{\sim}{ }   &  -\frac{\sin a\pi}{a\pi} 2\pi^{3/2}  \left(\frac{2}{\sqrt{-3}}\right)^{1/2} \left(\frac{M+a}{2\pi i}\right)^{3/2}\exp\left(\frac{M+a}{2\pi i}\times i \Vol(\mathbb{S}^3 \backslash 4_1)\right)\notag\\
\stackrel[M \to \infty]{\sim}{ }  &  \frac{\sin a\pi}{a\pi} \frac{1}{3^{1/4}} (M+a)^{3/2} \exp\left(\frac{(M+a)\Vol(\mathbb{S}^3 \backslash 4_1)}{2\pi}\right),
\end{align}
where we take $\ds \frac{\sin a\pi}{a\pi}=1$ when $a=0$.
\end{theorem}

Next we consider the case where $a$ and $M$ satisfies some limiting constraints. Theorem~\ref{mainthm2} below corresponds to the case where $s \sim 1$.

\begin{theorem}\label{mainthm2} Let $q = \exp(\frac{2\pi i}{N+\frac{1}{2}})$, i.e. $a = N-M+\frac{1}{2}$. Let $\ds s = \lim_{N \to \infty}\frac{M}{N+1/2}$. Then
\begin{enumerate}
\item there exists some $\delta>0$ such that for any $1-\delta < s < 1$, we have
\begin{align}
J_{M}\left(4_1;q \right)
\stackrel[M \to \infty]{\sim}{ } &\frac{1}{i\sin(s\pi)}
 (N+\frac{1}{2})^{1/2} \frac{\sqrt{2\pi}\exp\left((N+\frac{1}{2})\tilde{\Phi}^{(s)}_{M}\left(z^{(s)}_{M}\right)\right)}{\sqrt{{\tilde{\Phi}}_{M}^{(s)''}(z^{(s)}_M)} }
\end{align}where
\begin{align*}
\tilde{\Phi}^{(s)}_{M}(z)
=&\frac{1}{2\pi i}\left[\operatorname{Li_{2}}\left(e^{-2\pi iz + 2\pi i\left(\frac{M}{N+\frac{1}{2}}\right)}\right) - \operatorname{Li_{2}}\left(e^{ 2\pi iz + 2\pi i\left( \frac{M}{N+\frac{1}{2}}\right)}\right)\right]+2\pi i\left( 1 - \frac{M}{N+1/2} \right)z
\end{align*}
and $z^{(s)}_M$ satisfies the equation
\begin{align}\label{saddeqn1}
\beta_{M}\omega^2 - (\beta_M^2 +1 - \beta_M)\omega + \beta_M = 0 ,
\end{align}
where $\ds \beta_M= e^{2\pi i(\frac{M}{N+\frac{1}{2}})}$ and $\ds \omega_M = e^{2\pi iz^{(s)}_{M}}$.
\item we have
$$\Re [\tilde{\Phi}^{(s)}_{M}(z^{(s)}_M)] = \frac{1}{2\pi} \Vol\left( \mathbb{S}^3 \backslash 4_1, 2\pi \left(1-\frac{M}{N+1/2}\right) \right), $$
where $ \Vol\left( \mathbb{S}^3 \backslash 4_1, 2\pi \left(1-\frac{M}{N+1/2}\right) \right)$ is the hyperbolic volume of the cone manifold with singularity the figure eight knot and cone angle $2\pi \left(1-\frac{M}{N+1/2}\right)$.
\end{enumerate}
\end{theorem}

Next, for $s \sim \frac{1}{2}$, we have
\begin{lemma}\label{mainlemma} There exists $\rho>0$ such that whenever $|s-\frac{1}{2}|<\rho$, we have
$$\lim_{N \to \infty} \frac{2\pi}{N+1/2}\log|J_M (4_1;q)| \leq k\Vol(\mathbb{S}^3 \backslash 4_1) $$ 
where $k\in(0,1)$ is some constant independent on $s$.
\end{lemma}

Lemma~\ref{mainlemma} implies that in order to study the asymptotics of the sum of the colored Jones polynomials, actually we only need to consider the sum of those with $s\sim 1$. Furthermore, the AEF for the sum can be found by the Laplace's method. As a result, by using Theorem~\ref{mainthm2}, we obtain the AEF for the Turaev-Viro invariants for the figure eight knot stated as follows.

\begin{theorem}\label{mainthm4} For any $r = 2N+1 > 3$, the AEF of the Turaev-Viro invariants of the figure eight knot complement is given by
\begin{align}
TV_{r}\left(\mathbb{S}^3 \backslash 4_1, e^{\frac{2\pi i}{r}}\right)
\stackrel[N \to \infty]{\sim}{ } & (\eta_{r}')^{2} \sum_{M: 1-\delta \leq s \leq 1}\left|\tilde{J}_{M}\left(4_1,e^{\frac{2\pi i }{N+\frac{1}{2}}}\right)\right|^2  \notag\\
\stackrel[N \to \infty]{\sim}{ } &   \left(\frac{1}{4}\right) \left(\frac{r}{2}\right)^{1/2} \left|\frac{2}{ \sqrt{-3}} \right|^{3/2} \exp \left(\frac{ r}{2\pi}\Vol(\mathbb{S}^3 \backslash 4_1)\right),
\end{align}
where $\dfrac{2}{ \sqrt{-3}}$ is the twisted Reidemeister torsion associated with the unique complete hyperbolic structure of $\mathbb{S}^3 \backslash 4_1$.
\end{theorem}

\subsection{Interpretation of the AEF for the colored Jones polynomials}

Here we give some comments on Theorem~\ref{mainthm2}. Note that if $s=1$ (e.g. $M=N$), we can see that as $M$ goes to $\infty$, $\tilde{\Phi}^{(1)}_M(z)$ tends to the function
$$  \tilde{\Phi}^{(1)}_\infty(z)=\frac{1}{2\pi i}\left[\operatorname{Li_{2}}\left(e^{-2\pi i z}\right) -  \operatorname{Li_{2}}\left(e^{ 2\pi i z}\right)\right]   .$$
Moreover, the saddle point equation ``tends to''
$$  \omega^2 - \omega + 1= 0    ,$$
for which the suitable solution is given by
\begin{align}\label{sadloc}
z^{(1)}_\infty =\frac{5}{6}\text{ and }\omega = \exp(2\pi i z^{(1)}_\infty ) = \exp \left(\frac{5\pi i}{3}\right)
\end{align}
Furthermore, the evaluation of the function $\tilde{\Phi}^{(1)}_\infty(z)$ at the point $z=\dfrac{5}{6}$ gives the hyperbolic volume of the figure eight knot:
$$ \tilde{\Phi}^{(1)}_\infty \left(\frac{5}{6}\right) = \frac{1}{2\pi i} \left[\operatorname{Li_{2}}\left(e^{-\frac{5\pi i}{3}}\right) -  \operatorname{Li_{2}}\left(e^{ \frac{5\pi i}{3}}\right)\right] =\frac{1}{2\pi} \Vol(\mathbb{S}^3 \backslash 4_1)  $$

This observation is consistent with our expectation that the growth rate of the colored Jones polynomials with $s \sim 1$ is close to the hyperbolic volume of the figure eight knot.

Moreover, the growth rate of the colored Jones polynomials with $s \sim 1$ is given by

\begin{align}
\tilde{\Phi}^{(s)}_{M}(z^{(s)}_M)
=&\frac{1}{2\pi i}\left[\operatorname{Li_{2}}\left(e^{-2\pi iz^{(s)}_M + 2\pi i\left(\frac{M}{N+\frac{1}{2}}\right)}\right) - \operatorname{Li_{2}}\left(e^{ 2\pi iz^{(s)}_M + 2\pi i\left( \frac{M}{N+\frac{1}{2}}\right)}\right)\right] \notag\\
& \qquad+2\pi i\left( 1 - \frac{M}{N+1/2} \right)z^{(s)}_M
\end{align}
with $z^{(s)}_M$ satisfies the equation
\begin{align}\label{saddeqn1}
\beta_{M}\omega_M^2 - (\beta_M^2 +1 - \beta_M)\omega_M + \beta_M = 0 ,
\end{align}
where $\ds \beta_M= e^{2\pi i(\frac{M}{N+\frac{1}{2}})}$ and $\ds \omega_M = e^{2\pi iz^{(s)}_{M}}$.

Note that Equation~(\ref{saddeqn1}) is equivalent to the equation
$$ \omega_M + \omega_M^{-1} = \beta_M + \beta_M^{-1} -1 $$

If we write $\omega = \omega_M$ and $B = \beta_M ^{-1}$, then the equation can be written as
\begin{align}
 \omega + \omega^{-1} = B + B^{-1} -1
\end{align}
and the value $\tilde{\Phi}^{(s)}_{M}(z^{(s)}_M)$ can be expressed in the form
\begin{align}
&\hspace{14pt} \tilde{\Phi}^{(s)}_{M}(z^{(s)}_M)\\
&= \frac{1}{2\pi i}\left[\operatorname{Li_{2}}\left(\omega^{-1}B^{-1}\right) - \operatorname{Li_{2}}\left(\omega B^{-1}\right)\right] +2\pi i\left( 1 - \frac{M}{N+1/2} \right)z^{(s)}_M  \notag\\
&= \frac{1}{2\pi i}\left[\operatorname{Li_{2}}\left(\omega^{-1}B^{-1}\right) - \operatorname{Li_{2}}\left(\omega B^{-1}\right)\right] +2\pi i(z^{(s)}_M - 1)\left( 1 - \frac{M}{N+1/2} \right) + 2\pi i \left(1- \frac{M}{N+1/2}\right) \notag\\
&=  \frac{1}{2\pi i}\left[\operatorname{Li_{2}}\left(\omega^{-1}B^{-1}\right) - \operatorname{Li_{2}}\left(\omega B^{-1}\right) +(2\pi i (z^{(s)}_M-1))(2\pi i (1-\frac{M}{N+1/2})\right] + 2\pi i \left(1- \frac{M}{N+1/2}\right)  \notag\\
&=  \frac{1}{2\pi i}\left[\operatorname{Li_{2}}\left(\omega^{-1}B^{-1}\right) - \operatorname{Li_{2}}\left(\omega B^{-1}\right) +\log \omega \log B\right] + 2\pi i \left(1- \frac{M}{N+1/2}\right) \notag\\
&=  \frac{1}{2\pi i}H( \omega, B) + 2\pi i \left(1- \frac{M}{N+1/2}\right)
\end{align}
where $H(x,y)$ is the function appear in \cite{HM04.2} given by
$$ H(x,y) = \operatorname{Li_{2}} ( x^{-1}y^{-1}) - \operatorname{Li_{2}}( x y^{-1}) + \log x \log y $$
and $\log$ is the principal logarithm. In particular, we have
\begin{align*}
\exp\left( \left(N+\frac{1}{2}\right)  \tilde{\Phi}^{(s)}_{M}(z^{(s)}_M)  \right)  
= \exp\left( \frac{N+1/2}{2\pi i}  H(\omega,B)  + \pi i  \right)
= - \exp\left( \frac{N+1/2}{2\pi i}  H(\omega,B) \right)
\end{align*}
By Theorem 1.2 and Remark 1.4 in \cite{HM04.1}, we have the second part of Theorem~\ref{mainthm2}.

Finally we compare our result with \cite{HM04.1}. Let $5/6 < r < 1$ be fixed. In \cite{HM04.1}, the growth rates of the $M$-th colored Jones polynomials evaluated at $\exp ( 2\pi i r /M)$ is computed (note that the $M$ here is the $N$ in \cite{HM04.1}). It depends on whether $r$ is irrational. The difference between rational and irrational $r$ is that for rational $r$, the equation $h_M(l) |_{q =\exp ( 2\pi i r /M) } =0$, where
\begin{align*}
h_M(l) |_{q =\exp ( 2\pi i r /M) } &= (q^{(M+l)/2} - q^{-(M+l)/2}) (q^{(M-l)/2} - q^{-(M-l)/2}) \\
&= -4\sin( (M+l)\pi r/M)\sin( (M-l)\pi r  /M),
\end{align*}
has an integer solution $l = M(1-r)/r \in (1,M/5)\cap \NN$ for certain choices of $M \in \NN$. In particular we have
$$ \prod_{j=1}^{L}h_M(l) |_{q=\exp(2\pi i r/M) } = 0 \times \text{ something that may have exponential growth} = 0$$
whenever $M(1-r)/r<L<M-1$.
It is natural to compare our Theorem~\ref{mainthm2} to Murakami's result with $r = \frac{M}{N+1/2}$. Consider the evaluation 
$$h_M(l)|_{\exp(2\pi i M/ (N+1/2))} = -4\sin\left((M+l)\pi/(N+1/2)\right)\sin\left( (M-l)\pi /(N+1/2)\right)$$
First of all, the number $\frac{M(1-r)}{r}$ becomses $(N+\frac{1}{2} - M) \notin \NN$ . Also we have $\dfrac{M\pm l}{N+1/2}\neq 1$. Furthermore, since $0<l<M<N+1/2$, we have
$$ 0< \frac{M+l}{N+1/2} < \frac{2N}{N+1/2} < 2 \text{ and } 0< \frac{M-l}{N+1/2} < \frac{M}{N+1/2}< 2$$
Hence the evaluation $h_M(l)|_{\exp(2\pi i M/ (N+1/2))}$ never vanishes for any $M,N,l \in \NN, 0<l<M<N+1/2$.

We suggest that this kind of vanishing phenomenon is the reason why the Turaev-Viro invariants and the Reshetikhin-Turaev invariants have different asymptotic behaviors at unitary roots and non-unitary roots (see for example \cite{Ga98}). Similar phenomenon can also be found in the evaluation of $M$-th colored Jones polynomials at ($M$+integer)th-root of unity.

\subsection{From the volume conjecture for colored Jones polynomials to the volume conjecture for Turaev-Viro invariants}

Finally we summarized the techniques used in this paper and try to relate the AEF of colored Jones polynomials and that of Turaev-Viro invariants from an analytical perspective.

It is suggested by H. Murakami \cite{HM11} that the AEF of the colored Jones polynomials can be expressed in the form of a contour integral of a function with the form $e^{N\Phi^{(1)}_\infty (z)}$ along some suitable contour $C$. Here the function $\Phi^{(1)}_\infty(z)$ is called the potential function of the knot $K$ and has the property that the evaluation of the potential function at some saddle point is equal to the complex volume of the knot $K$.

In this paper, the authors generalize the above approach by using a one-parameter family version of the saddle point approximations. This gives a family of potential functions $\tilde{\Phi}^{(s)}_M (z)$ for $\ds s=\lim_{N \to \infty} \frac{M}{N+1/2} \sim~1$. From this we obtained the AEF as stated in Theorem~\ref{mainthm1},~\ref{mainthm1.5} and~\ref{mainthm2}.

This idea can be summarized by the following conjecture:
\begin{conjecture}\label{paravc}
For any hyperbolic knot $K$ and two integers $M$ and $N$, we denote $\ds s = \lim_{N \to \infty} \frac{M}{N+1/2}$. Then there exists a small neighborhood $U\subset \RR$ of $s=1$ such that for any $M$, $N$ with $s \in U$, we have a holomorphic function $\Phi^{(s)}_{M}(z)$ satisfying the following properties:
\begin{enumerate}
\item the holomorphic function $\Phi^{(1)}_{N}(z)$ gives the potential function of the knot $K$ by taking limit $N\to \infty$, i.e.
$$  \lim_{N \to \infty} \Phi^{(1)}_N (z) = \Phi^{(1)}_\infty (z) .$$
Furthermore, denote $z^{(1)}_0$ to be the non-degenerate saddle point of the potential function $\Phi^{(1)}_\infty (z)$ that gives the complex volume of the knot $K$, i.e.
$$ \frac{d}{dz}\Phi^{(1)}_\infty (z^{(1)}_\infty) = 0 \quad \text{and} \quad \frac{d^2}{dz^2} \Phi^{(1)}_\infty (z^{(1)}_\infty) \neq 0.$$
Then there exist a smooth choices of saddle point $z_M^{(s)}$ of the family of potential functions such that
\begin{enumerate}
\item the points $z_M$ satisfy the saddle point equations and they are non-degenerate, i.e.
$$ \frac{d}{dz} \Phi^{(s)}_M (z^{(s)}_M) = 0 \quad \text{and} \quad \frac{d^2}{dz^2} \Phi^{(s)}_M(z^{(s)}_M) \neq 0 ;$$
\item as $M \to N$, we have
 $$z^{(s)}_M \to z^{(1)}_\infty \quad \text{and} \quad \Phi^{(s)}_{M}(z^{(s)}_M) \to \Phi^{(1)}_\infty (z^{(1)}_\infty) = \frac{1}{2\pi}[\Vol (K) + i \CS (K)].$$
\end{enumerate}
\item the family of potential functions determine the AEF of the colored Jones polynomials in the following way:
\begin{align*}
J_{M} (K, e^{\frac{2\pi i}{N+1/2}}) \stackrel[M \to \infty]{\sim}{ } \text{(constant)} \times \left( \frac{N}{2\pi i} \right)^{3/2}\frac{ \exp \left( (N+1/2) \times \Phi_M^{(s)}(z^{(s)}_M) \right)}{\sqrt{ \frac{d^2}{dz^2} \Phi_M^{(s)} (z^{(s)}_M)}}.
\end{align*}
\end{enumerate}
\end{conjecture}

In particular Conjecture~\ref{paravc} is true for $K = 4_1$.

Similar idea can be applied to the study of AEF of the Turaev-Viro invariants. For example, naively the TV invariants of the figure eight knot complement can be thought of as a double integral over a suitable surface, with one integral corresponding to the sum inside each colored Jones polynomials and the other integral corresponding to the sum over all the colored Jones polynomials. Note that this idea can be found in \cite{AH06} where the AEF for RT invariant is studied. So it is natural to think that the volume conjecture for TV invariants is equivalent to the 2-dimensional saddle point approximations over a suitable surface. Note that in general the sum inside each colored Jones polynomials could be multi-index $k > 1$. Naively the volume conjecture for TV invariants is equivalent to the $(k+1)$ dimensional saddle point approximations over a suitable $(k+1)$ dimensional manifold.

In this paper our approach is different from that in \cite{AH06}. We break the 2-dimensional saddle point approximations into iterated 1-dimensional saddle point approximations, by first finding the AEF of colored Jones polynomials with $s\sim 1$ parametrized by the ratio $\ds \frac{M}{N+1/2}$ and then apply the saddle point approximations again along the parameter $s$.

Besides, from the development of the volume conjecture, we expect that the AEF of the colored Jones polynomials should be related to the character variety of the knot complement. By Mostow rigidity, there exists a unique point on the character variety which corresponds to the complete hyperbolic structure. The classical volume conjecture is about the topology (Reidemeister torsion) and geometry (hyperbolic volume) at this point. From an analytical perspective, the volume conjecture corresponds to the classical saddle point approximations.

In this paper we study the AEF of the $M$-th colored Jones polynomials evaluated at $(N+1/2)$-th root of unity. By introducing the limiting ratio $\ds s = \lim_{N \to \infty} \frac{M}{N+1/2}$, the AEF of $J_M (4_1, \exp(2\pi i/(N+1/2)))$ has been found out for $s\sim 1$. Moreover, the real part of the exponential growth rate coincides with the volume of the cone manifold. Therefore, the number $s$ can be thought of a parametrization of the points on the variety. Using the idea of the classical case, the AEF of the colored Jones polynomials with limiting ratio $s$ should also capture the same types of topological and geometrical information. From an analytical perspective, this kind of `volume conjecture' corresponds to the one-parameter family version of saddle point approximations.

In our study about the figure eight knot, although we cannot find out the explicit AEF for the case where $s \sim 1/2$, we are able to show that their exponential growth rate is less than the hyperbolic volume of the figure eight knot complement in the sense of Lemma~\ref{mainlemma}. This can be explained as follows. By the work of Thurston \cite{T79}, we know that the hyperbolic volume of the manifold with complete hyperbolic structure is strictly greater than that with incomplete hyperbolic structure. Hence, if $s$ is not close to 1 (that means the point is away from the point with complete structure), then the exponential growth rate (volume of the manifold at that point) is strictly smaller and hence can be ignored.

From above discussion, we expect that this kind of phenomenon is true for any hyperbolic knot. More precisely the conjecture can be stated as follows.

\begin{conjecture}\label{s1dom}
In the context of Conjecture~\ref{paravc}, for any hyperbolic knot $K$, the sum of colored Jones polynomials with $s \in U$ dominates the ones with $s \notin U$, i.e.
$$ \sum_{M: s \notin U}\left|\tilde{J}_{M}\left(K,e^{\frac{2\pi i }{N+\frac{1}{2}}}\right)\right|^2  = o \left( \sum_{M: s \in U}\left|\tilde{J}_{M}\left(K ,e^{\frac{2\pi i }{N+\frac{1}{2}}}\right)\right|^2  \right),$$
where $a_N = o (b_N)$ if and only if $\ds \lim_{N \to \infty} \frac{a_N}{b_N} = 0 $.
\end{conjecture}
With Conjecture~\ref{s1dom}, our approach of finding the AEF of the TV invariant can be formulated as the following conjecture:

\begin{conjecture}
In the context of Conjecture~\ref{paravc} and~\ref{s1dom}, for any hyperbolic knot $K$, we can find a function $\Phi(s,z) : D \subset U \times \CC \to \CC$ holomorphic in $z$ such that
\begin{enumerate}
\item the holomorphic function $\Phi(s,z)$ recovers the potential functions introduced in Conjecture~\ref{paravc}, i.e.
$$ \Phi\left( \frac{M}{N+1/2}, z \right) = \Phi^{(s)}_M(z) \quad \text{and}\quad \Phi(1,z) = \Phi^{(1)}_\infty(z);$$
\item there exists a smooth choice of non-degenerate saddle points $z(s)$ such that for each $s \in D$,
$$ \frac{d}{dz} \Phi(s, z(s)) = 0  \quad \text{and} \quad \frac{d^2}{dz^2} \Phi(s, z(s)) \neq 0;$$
\item the AEF of the Turaev-Viro invariants is given by
\begin{align*}
TV_r (K)
&\hspace{8pt} = \hspace{4pt}  TV_r (\mathbb{S}^3 \backslash K, e^{\frac{2\pi i}{r}} ) \\
&\hspace{8pt} = \hspace{4pt} (\eta'_r)^2 \sum_{1 \leq M \leq N} \left| \tilde{J}_M \left(K, e^{\frac{2\pi i}{N+1/2}} \right) \right|^2 \\
&\stackrel[N \to \infty]{\sim}{ }  (\eta'_r)^2 \sum_{M, s \in U} \left| \tilde{J}_M \left(K, e^{\frac{2\pi i}{N+1/2}} \right) \right|^2 \\
&\stackrel[N \to \infty]{\sim}{ }  \left(\frac{1}{4} \right) \left( \frac{r}{2} \right)^{1/2} \left|T(K) \right|^{3/2} \exp \left(\frac{ r}{2\pi } \times  (\Vol(\mathbb{S}^3 \backslash K) \right),
\end{align*}
where $T(K)$ and $\Vol(K)$ are the twisted Reidemeister torsion and the hyperbolic volume associated with the unique complete hyperbolic structure of $\mathbb{S}^3 \backslash K$ respectively.
\end{enumerate}
\end{conjecture}

\subsection{Organization}

In Section 2 we will outline the proof of the main theorems. In order to focus on the key ideas, the proofs of the technical statements will be collected in Section 3.

\section{Proof Outline of the Main Theorem}\label{sec2}

This section is divided into three parts. The first part aims to prove Theorem~\ref{mainthm1} and illustrate the techniques used in \cite{WA17} and \cite{HM13}. We will show the AEF for the colored Jones polynomials at $(M+a)$-th root of unity with fixed $a\geq 0$. The AEF will then be generalized to the case where $a>0$ satisfies some limiting relation with $M$. This gives the proof of Theorem~\ref{mainthm2}. After that we prove Lemma~\ref{mainlemma}. Finally, we apply the AEF's obtained in part two to prove Theorem~\ref{mainthm4}.

\subsection{AEF for the colored Jones polynomials around $(M+a)$-th root of unity with fixed $a\geq 0$}

Let $a\geq 0$ be fixed. We are going to consider the asymptotic behavior of the $M$-th colored Jones polynomials around $(M+a)$-th root of unity, i.e. $q = \exp(\frac{2\pi i + u}{M+a})$ with $0 \leq u < \log((3+\sqrt{5})/2)$.

Recall that the formula of colored Jones polynomials, the definition of quantum dilogarithm and its functional equation are given as follows:
\begin{enumerate}
\item For $M \in \NN$, \begin{align*} 
 J_{M}(4_{1};q) &=  1+\sum_{k=1}^{M-1} \prod_{l=1}^{k}\left(q^{\frac{M-l}{2}} - q^{-\frac{M-l}{2}} \right)\left(q^{\frac{M+l}{2}} - q^{-\frac{M+l}{2}} \right) \\
 &= 1+\sum_{k=1}^{M-1} q^{-k M} \prod_{l=1}^{k}\left(1-q^{M-l}\right)\left(1-q^{M+l}\right)\,
 \end{align*}
 \item Fix $\gamma \in \CC$ with $\Re(\gamma)>0$. Then for any $|\operatorname{Re}(z)|< \pi + \operatorname{Re}(\gamma)$, the quantum dilogarithm function is defined to be
$$S_{\gamma}(z) = \exp\left(\frac{1}{4}\int_{C_{R}}\frac{e^{zt}}{\sinh(\pi t)\sinh(\gamma t)}\frac{dt}{t}\right)  , $$
where $C_R = (-\infty, -R] \cup \Omega_R \cup [R, \infty)$ with $\Omega_R = \{ R e^{i (\pi -t)}\text{ $|$ }0\leq s \leq \pi \}$ for $0<R<\min \{\pi / |\gamma|, 1\}$.
\item For $| \operatorname{Re}(z) |  < \pi$, the quantum dilogarithm satisfies the functional equation:
$$(1+e^{iz})S_{\gamma}(z+\gamma)=S_{\gamma}(z-\gamma) $$
\end{enumerate}

Using the functional equation of the quantum dilogarithm, one may extend the definition of quantum dilogarithm to any complex number $z$ with
$$\Re(z) \neq \pi + 2m\Re(\gamma) \text{ and } \Re(z) \neq -\pi - 2m'\Re(\gamma)$$ for any $m,m' \in \NN$.

Now we are going to obtain the AEF of the colored Jones polynomials at $(M+a)$-th root of unity with $a>0$, $a \notin \NN$. In fact we are going to find out the AEF around the root of unity, i.e. $q=\exp(\frac{2\pi i + u}{M+a})$. Then by taking $u=0$ we can get our desired result.

Applying the functional equation of the quantum dilogarithm with the values
$$\ds \gamma=\frac{2\pi - iu}{2(M+a)}\,,\quad\xi = 2\pi i + u \quad\text{and}\quad z=\pi - iu -2(l+a)\gamma $$
and observing that $\ds \frac{\xi}{M+a}=2i\gamma $, we have
\begin{align}\label{F1}
\prod_{l=1}^{k}\left(1-e^{\frac{M-l}{M+a}\xi}\right) = \frac{S_{\gamma}(\pi - iu - (2(k+a)+1)\gamma)}{S_{\gamma}(\pi - iu - (2a+1)\gamma)}
\end{align}

Similarly, putting $z=-\pi -iu +2(l-a)\gamma$, we have
\begin{align}\label{F2}
\prod_{l=1}^{k}\left(1-e^{\frac{M+l}{M+a}\xi}\right) = \frac{S_{\gamma}(-\pi - iu + (1-2a)\gamma )}{S_{\gamma}(-\pi - iu + (2(k-a)+1)\gamma)}
\end{align}

Furthermore, we split the colored Jones polynomials into two parts:
\begin{align*}
 J_{M}(4_{1};q)
= &1+\sum_{k=1}^{M-1} q^{-k M} \prod_{l=1}^{k}\left(1-q^{M-l}\right)\left(1-q^{M+l}\right) \\
= &\left[ 1 +  \sum_{k=1}^{\lceil a \rceil -1} q^{-k M} \prod_{l=1}^{k}\left(1-q^{M-l}\right)\left(1-q^{M+l}\right) \right]\\
&\qquad+  \sum_{k=\lceil a \rceil }^{M-1} q^{-k M} \prod_{l=1}^{k}\left(1-q^{M-l}\right)\left(1-q^{M+l}\right)
\end{align*}

Overall from (\ref{F1}) and (\ref{F2}) we have

\begin{align}\label{Jsplit}
 J_{M}(4_{1};q)
= &\left[ 1 +  \sum_{k=1}^{\lceil a \rceil -1} q^{-k M} \prod_{l=1}^{k}\left(1-q^{M-l}\right)\left(1-q^{M+l}\right) \right]\notag\\
&\qquad+ \frac{S_{\gamma}(-\pi -iu - (2a-1)\gamma)}{S_{\gamma}(\pi -iu - (2a+1)\gamma)}\notag\\
&\qquad\qquad\qquad\times  \sum_{k=\lceil a \rceil }^{M-1} e^{-\frac{kM\xi}{M+a}} \frac{S_{\gamma}( \pi -iu - (2k+2a+1)\gamma)}{S_{\gamma}(-\pi -iu + (2k-2a+1)\gamma)}
\end{align}

Define
\begin{align*}
g_{M}(z) = \exp\left(-(M+a)(u-\frac{a\xi}{M+a})z\right) \frac{S_{\gamma}(\pi - iu + i\xi z  + i\xi (\frac{a}{M+a}))}{S_{\gamma}(-\pi - iu - i \xi z + i\xi (\frac{a}{M+a}))}
\end{align*}

Now we want to find an integral expression for $J_{M}(4_1; q)$ by using residue theorem.
Since $S_{\gamma}(z)$ is defined for $|\operatorname{Re}(z)|<\pi + \operatorname{Re}(\gamma)$ and $\operatorname{Re}(\gamma)>0$, one may check that $g$ is well-defined and analytic on the domain (i.e. open, connected) $D$ where
\begin{align*}
&D \\
=&\left\{
x+iy \in \mathbb{C}~\left|
\begin{array}{c}
- 2\pi \left( x+\frac{a}{M+a} \right) - \operatorname{Re}(\gamma)< uy < 2\pi\left(1-\left( x+\frac{a}{M+a} \right) \right) +\operatorname{Re}(\gamma) \\
- 2\pi \left( x-\frac{a}{M+a} \right) -\operatorname{Re}(\gamma)< uy < 2\pi\left(1-\left( x-\frac{a}{M+a} \right) \right) +\operatorname{Re}(\gamma)  \\
\end{array}\right.\right\}\\
&\\
=&\left\{
  x+iy \in \mathbb{C}~\left| \scalebox{1.05}{$- 2\pi \left( x-\frac{a}{M+a} \right) - \operatorname{Re}(\gamma)< uy <2\pi\left(1-\left( x+\frac{a}{M+a} \right) \right) +\operatorname{Re}(\gamma)$} \right.\right\}
\end{align*}

Next, let $\epsilon =\dfrac{2a+\frac{1}{2}}{2(M+a)}$. Consider the contour $C(\epsilon) = C_{+}(\epsilon)\cup C_{-}(\epsilon)$ with the polygonal lines $C_{\pm}(\epsilon)$ defined  by
\begin{align*}
C_{+}(\epsilon)&:\hspace{3em} 1-\epsilon \rightarrow 1- \frac{u}{2\pi} -\epsilon +i \rightarrow -\frac{u}{2\pi} + \epsilon +i \rightarrow \epsilon \\
C_{-}(\epsilon)&:\hspace{3em} \epsilon \rightarrow \epsilon + \frac{u}{2\pi} -i \rightarrow 1 - \epsilon+\frac{u}{2\pi} -i \rightarrow 1-\epsilon
\end{align*}
Note that for $k=\lceil a \rceil,\lceil a \rceil+1,\lceil a-1 \rceil+2, \dots, M-1$, the singularities $\dfrac{2k+1}{2(M+a)}$ of the function $z \mapsto\tan((M+a)\pi z)$  lie in $D$. This is the reason why we need to split $J_{M}(4_{1};q)$ into two parts. Using Residue Theorem, we may express the colored Jones polynomials as
\begin{align}\label{Jint1}
&J_{M}(4_{1};q)
= \left[ 1 +  \sum_{k=1}^{\lceil a \rceil -1} q^{-k M} \prod_{l=1}^{k}\left(1-q^{M-l}\right)\left(1-q^{M+l}\right) \right] + \notag\\
&\frac{S_{\gamma}(-\pi -iu - (2a-1)\gamma)}{S_{\gamma}(\pi -iu - (2a+1)\gamma)} \frac{(M+a)i\exp(\frac{u}{2}-\frac{a\xi}{2(M+a)})}{2} \int_{C(\epsilon)} \tan\left( (M+a)\pi z \right)g_{M} (z) dz
\end{align}

Note that as $M$ goes to infinity, the first part of $J_M (4_1;q)$ grows at most polynomially. So it suffices to consider the large $M$ behavior of the second part. In order to estimate the integral, let
$$ G_{\pm}(M,\epsilon) = \displaystyle
\int_{C_{\pm}(\epsilon)}\tan((M+a)\pi z)g_{M}(z)dz\,.$$

Then one may rewrite
\begin{align}\label{Jint2}
&J_{M}(4_{1};q)
= \left[ 1 +  \sum_{k=1}^{\lceil a \rceil -1} q^{-k M} \prod_{l=1}^{k}\left(1-q^{M-l}\right)\left(1-q^{M+l}\right) \right] \notag \\
&+ \frac{S_{\gamma}(-\pi -iu - (2a-1)\gamma)}{S_{\gamma}(\pi -iu - (2a+1)\gamma)}\times \frac{(M+a)i\exp(\frac{u}{2}-\frac{a\xi}{2(M+a)})}{2}(G_{+}(M,\epsilon) + G_{-}(M,\epsilon))
\end{align}
The integral in $G_{\pm}$ may be splited by adding and subtracting the same term as follows,
\begin{align*}
 G_{\pm}(M,\epsilon) = \pm i \int_{C_{\pm}(\epsilon)} g_{M}(z)dz + \int_{C_{\pm}(\epsilon)} (\tan((M+a)\pi z) \mp i)g_{M}(z) dz
\end{align*}
The second term can be controlled by the following analogue of Proposition 2.2 in \cite{WA17}.
\begin{proposition}\label{tanapp}
There exists a constant $K_{1,\pm}$ independent of $M$ and $\epsilon$ such that
$$ \left|   \int_{C_{\pm}(\epsilon)} (\tan((M+a)\pi z) \pm i)g_{M}(z) dz \right| < \frac{K_{1,\pm}}{M+a}\,.$$
\end{proposition}

To approximate $g_{M}$, define a function
$$\tilde{\Phi}_{M}(z)=\frac{1}{\xi}\left[\operatorname{Li_{2}}\left(e^{u-\left(z+\frac{a}{M+a}\right)\xi}\right) -
 \operatorname{Li_{2}}\left(e^{u + \left(z - \frac{a}{M+a}\right)\xi}\right)\right] - \left(u - \frac{a\xi}{M+a}\right)z $$

Since $\operatorname{Li_{2}}$ is analytic in $\mathbb{C}\backslash [1,\infty)$, by considering the region where
$$\im \left(u - \left(z + \frac{a}{M+a}\right)\xi \right), \im \left(u + \left(z - \frac{a}{M+a}\right)\xi \right)  \in (-2\pi, 0 ),$$
 one may verify that the function $\Phi_{M}(z)$ is analytic in the region
$$ D' = \left\{ x+iy \in \mathbb{C} \left| -2\pi \left(x - \frac{a}{M+a} \right) < uy < 2\pi \left(1-\left(x+\frac{a}{M+a}\right)\right) \right.\right\} \subset D$$

Note that the contour $C(\epsilon)$ and the poles of $\tan((M+a)\pi z)$ lie inside $D'$. The following result is an analogue of Proposition 2.3 in \cite{WA17}.

\begin{proposition}\label{exprepn}
Let $p(\epsilon)$ be any contour in the parallelogram bounded by $C(\epsilon)$ connecting from~$\epsilon$ to~$1-\epsilon$, then there exists a constant $K_{2}>0$ independent of $M$ and $\epsilon$ such that
\begin{align*}
\lefteqn{ \left| \int_{p(\epsilon)} g_{M}(z) dz - \int_{p(\epsilon)} \exp( (M+a) \tilde{\Phi}_{M}(z)) dz \right| \,\leq }\\
&\qquad\qquad \frac{K_{2} \log (M+a)}{M+a}
\max_{\omega \in p(\epsilon)} \left\{ \exp ((M+a) \operatorname{Re} \tilde{\Phi}_{M}(z)) \right\}
\end{align*}
\end{proposition}
Since $\tilde{\Phi}_{M}(z)$ is analytic on $D'$, by Cauchy's theorem
\begin{align*}
\int_{C_{+}(\epsilon)} \exp\left((M+a)\tilde{\Phi}_{M}(z)\right)dz = &-\int_{C_{-}(\epsilon)} \exp\left((M+a)\tilde{\Phi}_{M}(z)\right)dz
\end{align*}

Define a new function $\Phi_M(z)$ by
$$\Phi_{M}(z)=\frac{1}{\xi}\left[\operatorname{Li_{2}}\left(e^{u-\left(z+\frac{a}{M+a}\right)\xi}\right) -
 \operatorname{Li_{2}}\left(e^{u + \left(z - \frac{a}{M+a}\right)\xi}\right)\right] - uz $$

The contour integral can be further expressed as
\begin{align*}
\int_{C_{-}(\epsilon)} \exp\left((M+a)\tilde{\Phi}_{M}(z)\right)dz = \int_{C_{-}(\epsilon)} \exp(a\xi z) \exp\left((M+a)\Phi_{M}(z)\right)dz
\end{align*}

To approximate the above two integrals, we need the following generalized saddle point approximations. The proof is similar to that of Theorem~2.4 at \cite{WA17}.

\begin{theorem}\label{FSA}{\em (One-parameter family version for the saddle point approximations)}
Let $\{ \Phi_{y}(z)\}_{y \in [0,1]}$ be a family of holomorphic functions smoothly depending on $y \in [0,1]$. Let $C(y,t): [0,1]^{2} \to \CC$ be a continuous family of contours with length uniformly bounded above by a fixed constant $L$, such that for each $y \in [0,1]$, $C(y,t)$ lies inside the domain of $\Phi_{y}(z)$, for which $z_{y}$ is the only saddle point along the of contour $C_{y}$ and $\max\operatorname{Re}\left[\Phi_{y}(z)\right]$ is attained at $z_{y}$. Further assume that $\left|\arg\left( \sqrt{-\frac{d^{2}\Phi_{0}}{dz^{2}}(z_{0})}\right)\right| < \pi /4$. Suppose we have an analytic function f(z) along the contour such that $f(z_{y}) \neq 0$ for an $y \in [0,1]$. Then for any sequence $\{y_{M} \}_{M \in \mathbb{N}}$ with $y_{M} \to 0$ as $M \to \infty$, we have the following generalized saddle point approximations:
\begin{align*}
&\int_{C(y_{M},t)} f(z) \exp ((M+a) \Phi_{y_{M}}(z)) dz \\
=  &\sqrt{\frac{2\pi}{(M+a)\left(-\frac{d^{2} \Phi_{y_{M}}}{dz^{2}} (z_{y_{M}})\right)}}\, f(z_{y_M})\exp((M+a) \Phi_{y_{M}}(z_{y_{M}})) \left(1 + O\left(\frac{1}{M+a}\right)\right)
\end{align*}
\end{theorem}

In our case, we have
\begin{align*}
\Phi_{M}(z)=&\frac{1}{\xi}\left[\operatorname{Li_{2}}\left(e^{u-\left(z+\frac{a}{M+a}\right)\xi}\right) -
 \operatorname{Li_{2}}\left(e^{u + \left(z - \frac{a}{M+a}\right)\xi}\right)\right] - uz \\
\Phi_{\infty}(z)=& \lim_{M \to \infty} \Phi_M(z) = \frac{1}{\xi}\left[\operatorname{Li_{2}}\left(e^{u-z\xi}\right) -
 \operatorname{Li_{2}}\left(e^{u + z\xi}\right)\right] - uz
\end{align*}
Note that our function $\Phi_\infty(z)$ is the function $\Phi(z)$ in \cite{HM13}.

\begin{lemma}\label{lemmaFSA}
The curves described in Theorem~\ref{FSA} exist.
\end{lemma}

By Theorem~\ref{FSA} and Lemma~\ref{lemmaFSA}, we have

\begin{theorem}\label{saddlepointapp}{\em (Large $M$ behavior of $\displaystyle\int_{C_{\pm}(\epsilon)}\exp(a\xi z) \exp\left((M+a)\Phi_{M}(z)\right)dz$ )\/}
Let $\ds z_{M}$ be the saddle point of $\ds \Phi_{M}$ inside the contour $C(\epsilon)$. Then
\begin{align*}
&\int_{C_{-}(\epsilon)}\exp(a\xi z) \exp\left((M+a)\Phi_{M}(z)\right)dz \\
\stackrel[M \to \infty]{\sim}{ }\ &  \exp(a\xi z_M)\frac{\sqrt{2\pi}\exp\left((M+a)\Phi_{M}\left(z_{M}\right)\right)}{\sqrt{(M+a)}\sqrt{- \frac{d^{2} \Phi_{M}}{dz^{2}}\left(z_{M}\right)}}
\end{align*}
\end{theorem}

Together with the following proposition, which provides a control on the right-hand side, the integral in Theorem~\ref{saddlepointapp} is ensured to have exponentially growth.
\begin{proposition}\label{positive}
$\operatorname{Re}\Phi_{M}\left(z_{M}\right)$ is positive for $0\leq u<\log((3+\sqrt{5})/2)$ and $M$ is large.
\end{proposition}

Combining the controls in Propositions~\ref{tanapp} and~\ref{exprepn} and Theorem~\ref{saddlepointapp}, we are able to estimate $G_{\pm}(M,\epsilon)$, namely,
\begin{align*}
\lefteqn{\lim_{M \to \infty} \left| \frac{G_{\pm}(M,\epsilon)}{\displaystyle\mp i \int_{C_{\pm}(\epsilon)} \exp(a\xi z) \exp\left((M+a)\Phi_{M}(z)\right) dz} -1 \,\right| \quad\leq }\\
&\qquad\qquad \frac{K_{1, \pm}}{\displaystyle (M+a)\left| \int_{C_{\pm}(\epsilon)} \exp(a\xi z) \exp\left((M+a)\Phi_{M}(z)\right) dz \right|}\quad + \\
&\qquad\qquad\qquad
\frac{K_{2} \log(M+a)}{M+a} \times \frac{\exp\left((M+a) \operatorname{Re} \tilde{\Phi}_{M}(z_{M})\right)}{\displaystyle\left| \int_{C_{\pm}(\epsilon)} \exp(a\xi z) \exp\left((M+a)\Phi_{M})(z) \right) dz \right|} \\
&\qquad\qquad\qquad\qquad\qquad\qquad\qquad\qquad\qquad \xrightarrow{M \to \infty} 0\,.
\end{align*}
Thus, up to this point, we can asymptotically express $J_M$ in terms of quantum dilogarithm and a contour integral involving exponential of $(M+a)\Phi_M$. That is,
\begin{align}\label{Jfor}
&J_{M}\left(4_1,e^{\xi/(M+a)}\right) \quad\stackrel[M \to \infty]{\sim}{ } \frac{S_{\gamma}(-\pi -iu - (2a-1)\gamma)}{S_{\gamma}(\pi -iu - (2a+1)\gamma)}\notag\\
&\qquad\qquad\times (M+a)\exp\left(\frac{u}{2}-\frac{a\xi}{2(M+a)}\right)
\int_{C_{-}(\epsilon)} \exp(a\xi z) \exp\left((M+a)\Phi_{M}(z)\right)dz\,
\end{align}

Moreover, let $z_0$ be the saddle point of $\Phi(z)$. Then we have (see p.200 of \cite{HM13})
\begin{align}\label{tor}
\lim_{M \to \infty} \frac{d^{2} \Phi_{M}}{dz^{2}}\left(z_{M}\right)
= \frac{d^{2}}{dz^{2}} \Phi_{\infty} \left(z_0 \right)
= \xi \sqrt{(2\cosh u+1)(2\cosh u-3)}
\end{align}

The asymptotic behavior of the ratio of the quantum dilogarithm is given by the following lemma.

\begin{lemma}\label{Sratio}
For $\ds \gamma = \frac{2\pi - iu}{2(M+a)}$ with $u\geq 0$, let $a = b +c$ for some $b \in \NN \cup \{0\}$, $c \in (0,1)$.
\begin{enumerate}
\item If $b \neq 0$ and $u\neq 0$, we have
\begin{align*}
\frac{S_{\gamma}(-\pi-iu - (2a-1)\gamma)}{S_{\gamma}(\pi-iu-(2a+1)\gamma)}
&\hspace{8pt} = \hspace{8pt} \frac{\exp(u-2c\gamma i)-1}{\exp(u-2a\gamma i)-1} \times \frac{S_{\gamma}(-\pi-iu - (2c-1)\gamma)}{S_{\gamma}(\pi-iu-(2c+1)\gamma)}\\
&\hspace{8pt} = \hspace{8pt}\frac{\exp(u-2c\gamma i)-1}{\exp(u-2a\gamma i)-1} \times \frac{\exp(u\pi / \gamma -2c\pi i) -1 }{\exp(u-2c\gamma i) -1}\\
&\stackrel[M \to \infty]{\sim}{ } \frac{\exp(2\pi iu (M+a) /\xi -2c\pi i)  }{\exp(u) -1}.
\end{align*}
\item If $u=0$, we have
\begin{align*}
\frac{S_{\gamma}(-\pi - (2a-1)\gamma)}{S_{\gamma}(\pi-(2a+1)\gamma)}
&\hspace{8pt} = \hspace{8pt}\frac{\exp(-2c\pi i) -1 }{\exp(-2a\gamma i)-1}\\
&\hspace{8pt} = \hspace{8pt}\frac{\exp(-a\pi i) \sin {a\pi}}{\exp(-a\gamma i)\sin (a\gamma )}\\
&\stackrel[M \to \infty]{\sim}{ } \exp(-a\pi i + a\gamma i) \left(\frac{\sin a \pi}{a\pi}\right)\left(M+a\right).
\end{align*}
\end{enumerate}
\end{lemma}

Since $b$ is a non-negative integer, we also have $\exp(2\pi a i) =\exp(2\pi c i)$. By Theorem~\ref{saddlepointapp}, (\ref{Jfor}), (\ref{tor}) and Lemma~\ref{Sratio}, we have
\begin{align}\label{ch0}
&J_{M}\left(4_1,q \right) \notag\\
 \stackrel[M \to \infty]{\sim}{ }&  \frac{e^{2\pi iu (M+a) /\xi -2a\pi i}  }{e^{u} -1}
 (M+a)^{1/2} \left( e^{u/2 - a\xi/(2(M+a))}\right) \exp(a\xi z_M)\notag\\
&\qquad\qquad\qquad\qquad\qquad \times\frac{\sqrt{2\pi}\exp\left((M+a)\Phi_{M}\left(z_{M}\right)\right)}{\sqrt{-\xi \sqrt{(2\cosh(u)+1)(2\cosh(u)-3)} }} \notag\\
 \stackrel[M \to \infty]{\sim}{ }  &  \frac{e^{2\pi iu (M+a) /\xi -2a\pi i}  }{e^{u} -1}
 \sqrt{\frac{-2}{\sqrt{((2\cosh(u)+1)(2\cosh(u)-3)} }}\left(\frac{M+a}{\xi}\right)^{1/2} \notag\\
 &  \qquad\qquad\qquad\qquad\qquad \times \sqrt{\pi}e^{u/2}\exp(a\xi z_M)  \exp\left((M+a)\Phi_{M}\left(z_{M}\right)\right) \notag \\
\stackrel[M \to \infty]{\sim}{ }    & \frac{e^{2\pi iu (M+a) /\xi -2a\pi i}  }{e^{u} -1}
  T(u)^{1/2}\left(\frac{M+a}{\xi}\right)^{1/2}\notag \\
&  \qquad\qquad\qquad\qquad\qquad \times \sqrt{-\pi}e^{u/2} \exp(a\xi z_M)\exp\left((M+a)\Phi_{M}\left(z_{M}\right)\right) \notag \\
\stackrel[M \to \infty]{\sim}{ }    & \frac{e^{2\pi iu (M+a) /\xi -2a\pi i}  }{2\sinh (u/2)}
  T(u)^{1/2}\left(\frac{M+a}{\xi}\right)^{1/2} \notag\\
&  \qquad\qquad\qquad\qquad\qquad \times \sqrt{-\pi} \exp(a\xi z_M)\exp\left((M+a)\Phi_{M}\left(z_{M}\right)\right)
\end{align}

In order to apply the saddle point approximations, we have to solve the equation
\begin{equation}\label{saddlepteqn}
\frac{d\Phi_{M}}{dz}(z)=0 \,.
\end{equation}
Recall that
\begin{align*}
\Phi_{M}(z)=&\frac{1}{\xi}\left[\operatorname{Li_{2}}\left(e^{u-\left(z+\frac{a}{M+a}\right)\xi}\right) -
 \operatorname{Li_{2}}\left(e^{u + \left(z - \frac{a}{M+a}\right)\xi}\right)\right] - uz \\
\frac{d}{d\mu}\operatorname{Li}_{2}(e^{\mu})
&= \operatorname{Li}_{1}(e^{\mu})=-\log(1-e^{\mu})
\end{align*}
The desired saddle point equation~(\ref{saddlepteqn}) can be rewritten as below,
\begin{align}\label{cancel}
\log(1-e^{u-(z+\frac{a}{M+a})\xi})(1-e^{u+(z-\frac{a}{M+a})\xi}) - u  &=0 \,,\notag\\
\intertext{which in turns becomes,}
 (1-e^{u-(z+\frac{a}{M+a})\xi})(1-e^{u+(z-\frac{a}{M+a})\xi})  &= e^{u}\,.
\end{align}
With $\ds A=e^{u}$, $\ds B= e^{\frac{a}{M+a}\xi}$ and $\ds w = e^{z\xi}$, the above equation is equivalent to
\begin{align}\label{quadeqn}
 AB \omega^2 - (A^2 + B^2 - AB^2)\omega + AB=0
\end{align}

\begin{remark}
By putting $B=1$ we obtained the quadratic equation appearing in p.200 of \cite{HM13}.
\end{remark}

Let $\omega_{M}$ be the solution for~$\omega$ inside the domain $C(\epsilon)$ and $e^{z_{M}\xi}=\omega_{M}$. Recall that $\Phi_{\infty}(z)$ is defined to be the limit of $\Phi_M (z)$ as $M$ goes to infinity. Note that we have $z_{M} \to z_0$ as $M \to \infty$.
The last step to establish Theorem~\ref{mainthm1} is to change $\Phi_{M}$ into $\Phi_{0}$. The estimation between them is given by the following lemma, which is direct consequence of L'Hospital rule.

\begin{lemma}\label{diff1} For any $z \in D'$,
\begin{align*}
\lim_{M \to \infty} (M+a)\left(\Phi_{M}(z) - \Phi_{\infty}(z)\right) =  a[\log(1-e^{u-z\xi}) - \log(1 - e^{u+z\xi})]
\end{align*}
\end{lemma}

From Equation~(3.1) in \cite{HM13} we know that $z_{\infty}=\dfrac{\varphi(u) + 2 \pi i}{ \xi}$, where $\ds \varphi(u) = \arccosh(\cosh(u) - 1/2)$. That means
\begin{align}\label{ch1}
&\hspace{28pt} \exp \left((M+a)\Phi_{M}\left(z_{M}\right)\right)  \notag \\
&\stackrel[M \to \infty]{\sim}{}
\exp\left(a\log(1-e^{u-z_M \xi}) - a\log(1 - e^{u+z_M \xi})\right)\exp\left((M+a)(\Phi_{\infty}(z_{M}))\right) \notag\\
&\stackrel[M \to \infty]{\sim}{}  \frac{(1-e^{u-z_M \xi})^a}{(1 - e^{u+z_M \xi)^a}} \exp\left((M+a)(\Phi_{\infty}(z_{M}))\right) \notag\\
&\stackrel[M \to \infty]{\sim}{}  \frac{(1-e^{u-z_\infty \xi})^a}{(1 - e^{u+z_\infty \xi})^a}  \exp\left((M+a)(\Phi_{\infty}(z_{M}))\right) \notag\\
&\stackrel[M \to \infty]{\sim}{}  \frac{(1-e^{u-\varphi(u)})^a}{(1 - e^{u+\varphi(u)})^a}  \exp\left((M+a)(\Phi_{\infty}(z_{M}))\right)
\end{align}

Using (\ref{quadeqn}), one can show that $z_{M} - z_\infty = O\left(\frac{1}{M+a}\right)$. Together with the fact that $z_{\infty}$ satisfies the equation $\left.\dfrac{d\Phi_{\infty}}{dz}\right|_{z_\infty}=0$, we have
\begin{lemma}\label{diff2}
$\ds \lim_{M \to \infty} (M+a)\left(\Phi_{\infty}(z_{M}) - \Phi_{\infty}(z_{\infty})\right) = 0$
\end{lemma}

As a result, (\ref{ch1}) becomes
\begin{align}\label{Mto0}
 \exp \left((M+a)\Phi_{M}\left(z_{M}\right)\right)
 \stackrel[M \to \infty]{\sim}{} \frac{(1-e^{u-\varphi(u)})^a}{(1 - e^{u+\varphi(u)})^a}   \exp\left((M+a)(\Phi_{\infty}(z_{\infty}))\right)
\end{align}

Altogether, by (\ref{ch0}) and (\ref{Mto0}), we have

\begin{align*}
&J_{M}\left(4_1,q \right) \\
\stackrel[M \to \infty]{\sim}{ }  & (\exp(2\pi iu (M+a) /\xi -2a\pi i))\frac{ \sqrt{-\pi} }{2\sinh (u/2)}
  T(u)^{1/2}\left(\frac{M+a}{\xi}\right)^{1/2}  \\
& \qquad \times \exp(a(\varphi(u)+2\pi i))\frac{(1-e^{u-\varphi(u)})^a}{(1 - e^{u+\varphi(u)})^a} \exp\left((M+a)(\Phi_{\infty}(z_{\infty}))\right) \\
\stackrel[M \to \infty]{\sim}{ }  & \exp(a\varphi(u))\frac{(1-\exp(u-\varphi(u)))^a}{(1 - \exp(u+\varphi(u)))^a} \\
&\qquad \times  \frac{ \sqrt{-\pi} }{2\sinh (u/2)} T(u)^{1/2}\left(\frac{M+a}{\xi}\right)^{1/2} \exp\left((M+a)(S(u))\right). \\
\end{align*}

This proves Theorem~\ref{mainthm1}.

Similarly, for $u=0$, from [Remark 3.6, \cite{HM13}] we have
$$\varphi(0)=\ds \frac{-5\pi i}{3} \text{ and } \Phi_\infty(z_\infty)=\frac{\Vol(\mathbb{S}^3 \backslash K)}{2\pi}.$$
Furthermore, we have
$$  \exp(a(\varphi (0)+2\pi i)) = \exp\left(\frac{a\pi i}{3}\right)\text{ and } \frac{(1-e^{-\varphi(0)})^a}{(1 - e^{\varphi(0)})^a}= \exp\left(\frac{2a\pi i}{3}\right).$$

As a result, by Theorem~\ref{saddlepointapp}, Lemma~\ref{Sratio}, (\ref{Jfor}), (\ref{tor}), (\ref{ch0}) and (\ref{Mto0}), we have the following AEF:
\begin{align*}
J_{M}\left(4_1,q \right)
\stackrel[M \to \infty]{\sim}{ }    &  - \frac{\exp(-a\pi i)}{\exp(-a\gamma i)} \left(\frac{\sin a \pi}{a\pi}\right) \times  \exp(-a\gamma i ) \left(M+a\right)
  T(0)^{1/2}\left(\frac{M+a}{2\pi i}\right)^{1/2} \\
&\times \exp(a(\varphi(0)+2\pi i))\frac{(1-e^{-\varphi(0)})^a}{(1 - e^{\varphi(0)})^a} \sqrt{-\pi} \exp\left((M+a)\Phi_{\infty}\left(z_{\infty}\right)\right) \\
\stackrel[M \to \infty]{\sim}{ }   & - \frac{\sin a\pi}{a\pi} 2\pi^{3/2}  \left(\frac{2}{\sqrt{-3}}\right)^{1/2} \left(\frac{M+a}{2\pi i}\right)^{3/2}\exp\left(\frac{M+a}{2\pi i}\times i \Vol(\mathbb{S}^3 \backslash 4_1)\right)\\
\stackrel[M \to \infty]{\sim}{ }  &  \frac{\sin a\pi}{a\pi} \frac{1}{3^{1/4}} (M+a)^{3/2} \exp\left(\frac{(M+a)\Vol(\mathbb{S}^3 \backslash 4_1)}{2\pi}\right)
\end{align*}

This proves Theorem~\ref{mainthm1.5}

\subsection{AEF for the colored Jones polynomials at $(N+\frac{1}{2})$-th root of unity with $s$ closes to $1$}

Now, we try to apply the arguments in previous subsection to the case where $a=a_{M}$ is a sequence in $M$. Namely, we consider the case where $a_{M} = N-M + \frac{1}{2}$ and $u=0$, with $N>M$ being a positive integer. Further assume that the limit $\ds s= \lim_{N \to \infty} \frac{M}{N+1/2}$ is close to 1. Then we can split the colored Jones polynomials as before. Note that the following arguments also work if we replace $\frac{1}{2}$ by any other number $c \in (0,1)$ under suitable modification.

We repeat the trick as in previous subsection. Take $\ds \epsilon = \frac{2a+\frac{1}{2}}{2(M+a)} $ and define the contour $C(\epsilon) = C_{+}(\epsilon)\cup C_{-}(\epsilon)$ with the polygonal lines $C_{\pm}(\epsilon)$ defined by
\begin{align*}
C_{+}(\epsilon)&:\hspace{3em} 1-\epsilon \rightarrow 1- \frac{u}{2\pi}-\epsilon +i \rightarrow -\frac{u}{2\pi} + \epsilon  +i \rightarrow \epsilon  \\
C_{-}(\epsilon)&:\hspace{3em} \epsilon \rightarrow \epsilon + \frac{u}{2\pi} -i \rightarrow 1  - \epsilon+\frac{u}{2\pi} -i \rightarrow 1-\epsilon
\end{align*}

Define $g_M(z)$ and $G(M,\pm \epsilon)$ as before. We have the following analogues of the Proposition~\ref{tanapp} and Proposition~\ref{exprepn}.

\begin{proposition}\label{tanapp2} There exists some $\eta_1 >0$ such that for any $s \in (1-\eta_1, 1]$, there exists a constant $K_{1,\pm}$ independent of $M$, $N$, $s$ and $\epsilon$ such that
$$ \left|   \int_{C_{\pm}(\epsilon)} (\tan((N+1/2)\pi z) \pm i)g_{M}(z) dz \right| < \frac{K_{1,\pm}}{N+1/2}\,.$$
\end{proposition}

\begin{proposition}\label{exprepn2}
Let $p(\epsilon)$ be any contour in the parallelogram bounded by $C(\epsilon)$ connecting from~$\epsilon$ to~$1-\epsilon$ Then there exists some $\eta_2 >0$ such that for any $s \in (1-\eta_2, 1]$, there exists a constant $K_{2}>0$ independent of $M$, $N$, $s$ and $\epsilon$ such that
\begin{align*}
\lefteqn{ \left| \int_{p(\epsilon)} g_{M}(z) dz - \int_{p(\epsilon)} \exp( (N+\frac{1}{2}) \tilde{\Phi}_{M}(z)) dz \right| \,\leq }\\
&\qquad\qquad \frac{K_{2} \log (N+1/2)}{N+1/2}
\max_{\omega \in p(\epsilon)} \left\{ \exp ((N+\frac{1}{2}) \operatorname{Re}  \tilde{\Phi}_{M}(z)) \right\}
\end{align*}
\end{proposition}

When $u=0$, our function $\tilde{\Phi}_M (z)$ is given by
\begin{align*}
\tilde{\Phi}_{M}(z)
&=\frac{1}{\xi}\left[\operatorname{Li_{2}}\left(e^{-\left(z+\frac{a}{M+a}\right)\xi}\right) -
 \operatorname{Li_{2}}\left(e^{ \left(z - \frac{a}{M+a}\right)\xi}\right)\right] +\frac{a\xi z}{M+a}\\
&= \frac{1}{2\pi i}\left[\operatorname{Li_{2}}\left(e^{-2\pi iz + 2\pi i\left(\frac{M}{N+\frac{1}{2}}\right)}\right) - \operatorname{Li_{2}}\left(e^{ 2\pi iz + 2\pi i\left( \frac{M}{N+\frac{1}{2}}\right)}\right)\right] \\
&\quad + 2\pi i\left( 1 - \frac{M}{N+1/2} \right)z
\end{align*}

Now since $a = a_M$ is no longer fixed, the limiting functions of $\tilde{\Phi}_M(z)$'s are different for each $s$. What we have discussed in previous subsection can be considered as a special case where $s=1$. In general we define the function $\tilde{\Phi}^{(s)}_{M}$ and the limiting function $\tilde{\Phi}^{(s)}_{\infty}$ by
\begin{align*}
\tilde{\Phi}^{(s)}_{M}(z)
=&\frac{1}{2\pi i}\left[\operatorname{Li_{2}}\left(e^{-2\pi iz + 2\pi i\left(\frac{M}{N+\frac{1}{2}}\right)}\right) - \operatorname{Li_{2}}\left(e^{ 2\pi iz + 2\pi i\left( \frac{M}{N+\frac{1}{2}}\right)}\right)\right] \\
& \qquad+2\pi i\left( 1 - \frac{M}{N+1/2} \right)z\\
\tilde{\Phi}^{(s)}_{\infty}(z)=&\frac{1}{2\pi i}\left[\operatorname{Li_{2}}\left(e^{-2\pi iz + 2\pi i s}\right)- \operatorname{Li_{2}}\left(e^{ 2\pi iz + 2\pi i s}\right)\right] +2 \pi i\left( 1 - s \right)z
\end{align*}

Let $z^{(s)}_{\infty}$ be the solution of the saddle point equation
$$ \frac{d}{dz} \tilde{\Phi}^{(s)}_{\infty}(z) = 0$$

Since $z^{(s)}_{\infty}$ and the contour depend continuously on $s$, there exists a positive real number $\zeta<\min\{\eta_1,\eta_2\}$ such that for any $1-\zeta < s \leq 1$, the saddle point $z^{(s)}_{\infty}$ lies inside the contour $C(\epsilon)$. From now on we consider those $M$ satisfying $1-\zeta < s < 1$.

As a result, from (\ref{Jfor}), Lemma~\ref{Sratio} and Theorem~\ref{FSA} we have
\begin{align}\label{Jalimit1}
J_{M}\left(4_1,q \right) 
& \stackrel[M \to \infty]{\sim}{ } \left[ 1 +  \sum_{k=1}^{\lceil a \rceil -1} q^{-k M} \prod_{l=1}^{k}\left(1-q^{M-l}\right)\left(1-q^{M+l}\right) \right] \notag\\
&\qquad+\frac{S_{\gamma}(-\pi - (2a-1)\gamma)}{S_{\gamma}(\pi-(2a+1)\gamma)}
 (N+\frac{1}{2})e^{-a\gamma i} \frac{\sqrt{2\pi}\exp\left((N+\frac{1}{2})\tilde{\Phi}^{(s)}_{M}\left(z^{(s)}_{M}\right)\right)}{\sqrt{(N+\frac{1}{2})}\sqrt{\tilde{\Phi}_{M}^{(s)''}(z^{(s)}_M)} }
\end{align}

The following proposition ensures that the second term grows exponentially.

\begin{proposition}\label{positive2}
We may choose $\zeta > 0$ such that for every $1-\zeta<s\leq 1$,
$\operatorname{Re}\tilde{\Phi}^{(s)}_{M}\left(z^{(s)}_{M}\right)$ is positive when $M$ is sufficiently large.
\end{proposition}

Note that since now $a$ depends on $M$, the first term is not a finite sum. To deal with this term, we need to following lemma, which follows easily from the arguments in Theorem~4.1 of \cite{DKY17}.

\begin{lemma}\label{uppbd}
Let $\ds g_{M}(k) = \prod_{l=1}^{k} \left|\left(q^{\frac{M-l}{2}} - q^{-\frac{M-l}{2}} \right)\left(q^{\frac{M+l}{2}} - q^{-\frac{M+l}{2}} \right) \right|$. For each $M$, let $k_{M} \in \{ 1,\dots, M-1 \}$ such that $g_M(k_M)$ achieves the maximum among all $g_M (k)$. Assume that $ \frac{M}{r}=\frac{M}{2N+1} \to d \in [0,\frac{1}{2}]$ and $\frac{k_M}{r} \to k_{d}$ as $r \to \infty$. Then we have
$$ \lim_{r \to \infty}\frac{1}{r} \log( g_M (k_M) ) = -\frac{1}{2\pi}\left( \Lambda(2\pi(k_d -d)) + \Lambda(2\pi(k_d +d)) \right) \leq \frac{\Vol(\mathbb{S}^3 \backslash 4_1)}{4\pi}.$$
Furthermore, the equality holds if and only if $(s=2d=1$ and $2k_d = \frac{5}{6})$ or $(s=2d=\frac{1}{2}$ and $2k_d = \frac{1}{3})$.
\end{lemma}

By Lemma~\ref{uppbd}, since $j_d$ depends continuously on $d$, there exists a small neighborhood of $\frac{1}{2}$ such that for any $d$ in that neighborhood, $j_d$ is close to $\frac{5}{12}$. In other word, we may choose a small $\zeta$ to ensures that for any $1-\zeta<s\leq 1$, the maximum terms among all $g_M(j)$ appears in the second summation. Furthermore, the growth rate of the first summation will then be strictly less than a multiple (a number in $(0,1)$) of the growth rate of the second one. As a result, the first summation decays exponentially when it is compared to the second one.

To conclude, from Proposition~\ref{tanapp2}, Proposition~\ref{exprepn2} and Theorem~\ref{FSA}, we have
\begin{align}\label{asymch1}
&J_{M}\left(4_1,q \right)\stackrel[M \to \infty]{\sim}{ } \frac{S_{\gamma}(-\pi - (2a-1)\gamma)}{S_{\gamma}(\pi-(2a+1)\gamma)}
 (N+\frac{1}{2})^{1/2}e^{-a\gamma i} \times \frac{\sqrt{2\pi}\exp\left((N+\frac{1}{2})\tilde{\Phi}^{(s)}_{M}\left(z^{(s)}_{M}\right)\right)}{\sqrt{\tilde{\Phi}_{M}^{(s)''}(z^{(s)}_M)} }
\end{align}

By Lemma~\ref{Sratio}, we have
$$ \frac{S_{\gamma}(-\pi - (2a-1)\gamma)}{S_{\gamma}(\pi - (2a+1)\gamma)} = \frac{-2}{e^{-2a\gamma i}-1}= \frac{1}{ie^{-a\gamma i}\sin(a\gamma)}   \stackrel[M \to \infty]{\sim}{ }  \frac{1}{i e^{-a\gamma i} \sin (s\pi)}$$

Altogether we have
\begin{align}
J_{M}\left(4_1,q \right)
\stackrel[M \to \infty]{\sim}{ } &\frac{1}{i\sin(a\gamma)}
 (N+\frac{1}{2})^{1/2} \frac{\sqrt{2\pi}\exp\left((N+\frac{1}{2})\tilde{\Phi}^{(s)}_{M}\left(z^{(s)}_{M}\right)\right)}{\sqrt{{\tilde{\Phi}}_{M}^{(s)''}(z^{(s)}_M)} } \label{normal} \\
\stackrel[M \to \infty]{\sim}{ } &\frac{1}{i\sin(s\pi)}
 (N+\frac{1}{2})^{1/2} \frac{\sqrt{2\pi}\exp\left((N+\frac{1}{2})\tilde{\Phi}^{(s)}_{M}\left(z^{(s)}_{M}\right)\right)}{\sqrt{{\tilde{\Phi}}_{M}^{(s)''}(z^{(s)}_M)} }\label{Jalimit1f}
\end{align}

Now we explore to the saddle point equation in more detail. By direct computation one can see that the saddle point equation is given by
\begin{align*}
\beta_{M}\omega_M^2 - (\beta_M^2 +1 - \beta_M)\omega_M + \beta_M = 0 ,
\end{align*}
where $\ds \beta_M= e^{2\pi i(\frac{M}{N+\frac{1}{2}})}$ and $\ds \omega_M = e^{2\pi iz^{(s)}_{M}}$. This is exactly Equation~(\ref{saddeqn1}) and we complete the proof of Theorem~\ref{mainthm2}.

Here we give a remark on (\ref{Jalimit1f}) that will be used to find the AEF for the TV invariant later. Note that the suitable solution of the saddle point equation is given by
$$ \omega_M = \frac{(\beta_M^2 +1 - \beta_M) - \sqrt{(\beta_M^2 + 1 - 3\beta_M)(\beta_M^2 + 1 + \beta_M)}}{2\beta_M} .$$

Let $\omega$ be the solution of the following equation
$$ \omega= \frac{(\beta^2 +1 - \beta) - \sqrt{(\beta^2 + 1 - 3\beta)(\beta^2 + 1 + \beta)}}{2\beta} $$
with $\ds \beta= e^{2\pi i s}$ and $\ds \omega = e^{2\pi i z(s)}$. Define the function $\Theta(s)$ by
\begin{align*}
\Theta(s)
&= \tilde{\Phi}^{(s)}_{\infty}(z(s))\\
&=\frac{1}{2\pi i}\left[\operatorname{Li_{2}}\left(e^{-2\pi iz(s) + 2\pi i s}\right)- \operatorname{Li_{2}}\left(e^{ 2\pi iz(s) + 2\pi i s}\right)\right] +2 \pi i\left( 1 - s \right)z(s).
\end{align*}

Then we can see that $\Theta(s)$ depends smoothly on $s$ with $\Theta(\frac{M}{N+1/2}) = \tilde{\Phi}^{(s)}_{M}\left(z^{(s)}_{M}\right)$.

Similarly, for the evaluation of the second derivative at the saddle point, define the function $\Xi(s)$ by
\begin{align*}
\Xi(s) = 2\pi i e^{2\pi i s} (e^{-2\pi i z(s)} - e^{2\pi i z(s)}).
\end{align*}

Then by using the property that $\dfrac{d}{dz} \tilde{\Phi}^{(s)}_{M} (z_M) = 0$, one can verify that $\Xi(\frac{M}{N+1/2}) =  \tilde{\Phi}^{(s)''}_{M}\left(z^{(s)}_{M}\right)$. Note that since $ z(1) = \frac{5}{6}$, we also have $\Xi(1) = 2\pi\sqrt{3}$.

\subsection{An upper bound for the AEF of the $M$-th colored Jones polynomials of the figure eight knot at $(N+\frac{1}{2})$-th root of unity with $s$ closes to $\frac{1}{2}$}

In this subsection we are going to find an upper bound for the AEF of the $M$-th colored Jones polynomials at $(N+\frac{1}{2})$-th root of unity with the condition that
$$ s=2d =\lim_{N \to \infty}\frac{M}{N+1/2} \in \left(\frac{1}{2} - \delta, \frac{1}{2} + \delta\right),$$
where $\delta>0$ is a small number that will be clarified later.

First of all, we apply the functional equation of the quantum dilogarithm with the values
$$\ds \gamma=\frac{2\pi }{2N+1}\,,\quad\xi = 2\pi i\quad\text{and}\quad z=2\left(M - \frac{N+\frac{1}{2}}{2} - l \right)\gamma $$
Note that $\ds \frac{\xi}{N+\frac{1}{2}}=2i\gamma $. Moreover, we have
\begin{align}\label{F3}
\prod_{l=1}^{k}\left(1-e^{\frac{M-l}{N+1/2}\xi}\right) = \frac{S_{\gamma}\left(2\left(M - \frac{N+1/2}{2}\right)\gamma - (2k+1)\gamma\right)}{S_{\gamma}\left(2\left(M - \frac{N+1/2}{2}\right)\gamma - \gamma\right)}
\end{align}
Similarly, putting $z=2\left(M - \frac{N+\frac{1}{2}}{2} + l \right)i\gamma$, we have
\begin{align}\label{F4}
\prod_{l=1}^{k}\left(1-e^{\frac{M+l}{N+1/2}\xi}\right) = \frac{S_{\gamma}\left(2\left(M - \frac{N+1/2}{2}\right)\gamma +\gamma\right)}{S_{\gamma}\left(2\left(M - \frac{N+1/2}{2}\right)\gamma + (2k+1)\gamma\right)}
\end{align}

From (\ref{F3}) and (\ref{F4}), we have

\begin{align}
J_{M}(4_1,q) 
&= 1 + \frac{S_{\gamma}(2(M-\frac{N+1/2}{2})\gamma + \gamma)}{S_{\gamma}(2 (M-\frac{N+1/2}{2})\gamma - \gamma)}\sum_{k=1}^{M-1}e^{-\frac{2\pi k Mi}{N+1/2} }\frac{S_{\gamma}(2(M-\frac{N+1/2}{2})\gamma - (2k+1)\gamma)}{S_{\gamma}(2 (M-\frac{N+1/2}{2})\gamma + (2k+1)\gamma)}
\end{align}

Define
\begin{align*}
g_{M}(z) = \exp\left( -2\pi M z\right) \frac{S_{\gamma}\left(2\pi\left(\frac{M}{N+1/2} - \frac{1}{2}\right) - 2\pi z \right)}{S_{\gamma}\left(2\pi\left(\frac{M}{N+1/2} - \frac{1}{2}\right) + 2\pi z\right)}
\end{align*}

Then we have

\begin{align}
J_{M}(4_1,q) 
&= 1 + e^{\frac{M\pi}{N+1/2}}\frac{S_{\gamma}(2(M-\frac{N+1/2}{2})\gamma + \gamma)}{S_{\gamma}(2 (M-\frac{N+1/2}{2})\gamma - \gamma)}\sum_{k=1}^{M-1} g_M\left( \frac{2k+1}{2N+1} \right)
\end{align}

Recall from Lemma~\ref{uppbd} that for $s\sim \frac{1}{2}$, the maximum value is attained at $2k_d = \frac{1}{3}$. By continuity, there exists $\delta>0$ such that whenever $|s-\frac{1}{2}|<\delta$, the maximum value is attained in $2k_d\in(\frac{7}{24},\frac{9}{24})$. Therefore, we only need to consider those $k$ satisfying $\frac{2k+1}{2N+1} \in C = [\frac{5}{24},\frac{11}{24}]$. 

Next, we follow the argument of the proof of Proposition 4.2 in \cite{TO16}. Let $N(C)$ be a small neighborhood of $C$ in $\RR$ such that 
$$ \frac{2k+1}{2N+1} \in C \iff \frac{2k+1}{2N+1}  \in N(C)$$
Define a smooth bump function $b_M: \RR \to \RR$ satisfying
\begin{align*}
b_M(x) =
\begin{cases}
1 & \text{if }s \in C\\
0 & \text{if } s \not\in N(C)
\end{cases}
\end{align*} 
and $b_M(x) \in [0,1]$ on $N(C)\backslash C$. Then we have a smooth function $F_M: \RR \to \RR^2$ defined by 
$$F_M(x) = b_M\left(\frac{2x+1}{2N+1}\right)g_M\left(\frac{2x+1}{2N+1}\right)$$
By the Poisson Summation Formula, we have
\begin{align}
\sum_{k : \frac{2k+1}{2N+1} \in C} g_M\left( \frac{2k+1}{2N+1} \right)
=\sum_{k \in \ZZ}  F_M\left( k \right) 
=\sum_{k \in \ZZ} \widehat{F_M}(k)
\end{align}
where 
\begin{align*}
\widehat{F_M}(k) 
&= \int_{\RR} F_M(v)e^{2\pi i k v} d v \\
&= \frac{1}{N+1/2}\int_{\RR} (-1)^k F_M \left( \frac{(2N+1)x-1}{2} \right) e^{2\pi ki (N+1/2) x} dx \\
&= \frac{(-1)^k}{N+1/2}\int_{\RR}  b_M(x)g_M(x)e^{2\pi ki (N+1/2) x} dx \\
&= \frac{(-1)^k}{N+1/2}\int_{C} g_M(x)e^{2\pi ki (N+1/2) x} dx + \frac{(-1)^k}{N+1/2}\int_{N(C)\backslash C} b_M(x)g_M(x)e^{2\pi ki (N+1/2) x} dx
\end{align*}

From the above discussion, since the maximum point lies on the first integral, the second integral can be ignored. Besides, recall that for $|\Re(z)| < \pi$, the quantum dilogarithm can be expressed by [Equation (4.2) in \cite{AH06}]
$$ S_{\gamma} (z) = \exp \left( \frac{1}{2i\gamma}\operatorname{Li_2}(-e^{iz}) + I_\gamma (z)   \right) =\exp \left( \frac{N+\frac{1}{2}}{2\pi i}\operatorname{Li_2}(-e^{iz}) + I_\gamma (z)   \right)  ,$$
Thus,
\begin{align*}
&g_M(w)e^{2\pi ki (N+1/2) w}\\
&= \exp\left( -2\pi M w \right) \frac{S_{\gamma}\left(2\pi\left(\frac{M}{N+1/2} - \frac{1}{2}\right) - 2\pi z \right)}{S_{\gamma}\left(2\pi\left(\frac{M}{N+1/2} - \frac{1}{2}\right) + 2\pi z\right)} \times e^{2\pi ki (N+1/2) w }\\
&= \exp\left( (N+1/2) \left( \Phi^{(s)}_M(w) + 2\pi i(k-1)w\right) \right) \\
&\qquad \times \exp\left( I_\gamma\left(2\pi \left(\frac{M}{N+1/2}-\frac{1}{2}\right) - 2\pi w\right) - I_\gamma\left(2\pi \left(\frac{M}{N+1/2}-\frac{1}{2}\right) + 2\pi w\right)\right),
\end{align*}
where we recall again that the potential function $\Phi^{(s)}_M(w)$ and its limiting function $\Phi^{(s)}_\infty(w)$ are given by
\begin{align*}
\tilde{\Phi}^{(s)}_{M}(z)
=&\frac{1}{2\pi i}\left[\operatorname{Li_{2}}\left(e^{-2\pi iz + 2\pi i\left(\frac{M}{N+\frac{1}{2}}\right)}\right) - \operatorname{Li_{2}}\left(e^{ 2\pi iz + 2\pi i\left( \frac{M}{N+\frac{1}{2}}\right)}\right)\right] \\
& \qquad+2\pi i\left( 1 - \frac{M}{N+1/2} \right)z\\
\tilde{\Phi}^{(s)}_{\infty}(z)=&\frac{1}{2\pi i}\left[\operatorname{Li_{2}}\left(e^{-2\pi iz + 2\pi i s}\right)- \operatorname{Li_{2}}\left(e^{ 2\pi iz + 2\pi i s}\right)\right] +2 \pi i\left( 1 - s \right)z
\end{align*}

Recall the lemma 3 in \cite{AH06} that there exist $A,B>0$ dependent only on $R$ such that if $|\operatorname{Re}(z)|<\pi$, we have
$$ |I_{\gamma}(z)| \leq A\left(\frac{1}{\pi-\operatorname{Re}(z)} + \frac{1}{\pi + \operatorname{Re}(z)}\right)|\gamma| + B(1+e^{-\operatorname{Im}(z)R})|\gamma| $$

In particular, the term 
$$\exp\left( I_\gamma\left(2\pi \left(\frac{M}{N+1/2}-\frac{1}{2}\right) - 2\pi w\right) - I_\gamma\left(2\pi \left(\frac{M}{N+1/2}-\frac{1}{2}\right) + 2\pi w\right)\right)$$
has no exponential growth. Next, for the term, 
$$\exp\left( (N+1/2) \left( \Phi^{(s)}_M(w) + 2\pi i(k-1)w\right) \right),$$
note that for $s = \frac{1}{2}$ and $w\in C$, the real part of $\Phi^{(s)}_\infty(w) + 2\pi i(k-1)w$ attains its maximum when $w=\frac{1}{3}$ with maximum value
$$\Re \left[\Phi^{(s)}_\infty\left(\frac{1}{3}\right) + 2\pi i(k-1)\left(\frac{1}{3}\right)\right] = \Re \Phi^{(s)}_\infty\left(\frac{1}{3}\right) = \frac{1}{2\pi} \Vol(\mathbb{S}^3 \backslash 4_1)  $$
Besides, by direct computation, we have
\begin{align*}
\left.\frac{d}{dw}  \left( \Phi^{(s)}_M(w) \right) \right\vert_{w =\frac{1}{3}} = \pi i
\end{align*}
By continuity, we can find $r \in (0, \frac{1}{48})$ and make $\delta$ smaller such that whenever $|w-\frac{1}{3}|<r$ and $|s-\frac{1}{2}|<\delta$,
$$ \frac{\partial}{\partial y} \Re \left( \Phi^{(s)}_M(w)  \right) = - \im \frac{d}{dw}  \left( \Phi^{(s)}_M(w) \right) \in \left( -\frac{7\pi}{6}, -\frac{5\pi}{6}\right)$$
In particular, we have 
$$ \frac{\partial}{\partial y} \Re \left( \Phi^{(s)}_M(w) + 2\pi i(k-1)w\right) = -  \im\frac{d}{dw}  \left( \Phi^{(s)}_M(w) + 2\pi i(k-1)w\right) 
\begin{cases}
> \dfrac{5\pi}{6} & \text{if } k \leq 0 \\
< -\dfrac{5\pi}{6} &\text{if } k \geq 1
\end{cases}$$
WLOG assume that the real part of the potential function is increasing on the left of the maximum point and decreasing on the right of the maximum point. For $k\leq 0$, we consider the contours 
$$\tilde C : \frac{5}{24} \xrightarrow{\tilde{C_1}} \frac{1}{3} - \frac{r}{2} \xrightarrow{\tilde{C_2}} \frac{1}{3} - \frac{r}{2} - \frac{r}{2}i \xrightarrow{\tilde{C_3}}  \frac{1}{3} + \frac{r}{2} - \frac{r}{2}i \xrightarrow{\tilde{C_4}} \frac{1}{3} + \frac{r}{2} \xrightarrow{\tilde{C_5}} \frac{11}{24},$$
By assumption, 
\begin{align*}
\Re\left( \Phi^{(s)}_M(w) + 2\pi i(k-1)w\right) &\leq \Re\Phi^{(s)}_M\left(\frac{1}{3}-\frac{r}{2}\right) \quad \text{ for $w \in \tilde{C_1}$}\\
\Re\left( \Phi^{(s)}_M(w) + 2\pi i(k-1)w\right) &\leq \Re\Phi^{(s)}_M \left(\frac{1}{3}+\frac{r}{2} \right)\quad \text{ for $w \in \tilde{C_5}$}
\end{align*}
Besides, since 
$$\frac{\partial}{\partial y} \Re \left( \Phi^{(s)}_M(w) + 2\pi i(k-1)w\right) >0,$$
we have
\begin{align*}
\Re\left( \Phi^{(s)}_M(w) + 2\pi i(k-1)w\right) &\leq \Re\Phi^{(s)}_M\left(\frac{1}{3}-\frac{r}{2}\right) \quad \text{ for $w \in \tilde{C_2}$}\\
\Re\left( \Phi^{(s)}_M(w) + 2\pi i(k-1)w\right) &\leq \Re\Phi^{(s)}_M \left(\frac{1}{3}+\frac{r}{2} \right)\quad \text{ for $w \in \tilde{C_4}$}
\end{align*}
Finally, for any $w = x - \frac{r}{2}i \in \tilde{C_3}$ we have
\begin{align*}
\Re\left( \Phi^{(s)}_M\left(x - \frac{r}{2}i\right) + 2\pi i(k-1)\left(x-\frac{r}{2}i\right)\right)
&\leq \Re\left( \Phi^{(s)}_M\left(x\right) + 2\pi i(k-1)\left(x\right)\right) + \left(\frac{5\pi}{6}\right)\left(\frac{r}{2} \right) \\
&\leq \frac{1}{2\pi} \Vol(\mathbb{S}^3 \backslash 4_1) - \frac{5\pi r}{12}
\end{align*}
Altogether, we have
$$  \int_C \exp\left( (N+1/2) \left( \Phi^{(s)}_M(w) + 2\pi i(k-1)w\right) \right) dw = \int_{\tilde C} \exp\left( (N+1/2) \left( \Phi^{(s)}_M(w) + 2\pi i(k-1)w\right) \right) dw $$
with 
$$ \Re \left( \Phi^{(s)}_M(w) + 2\pi i(k-1)w\right) \leq \max\left\{ \frac{1}{2\pi} \Vol(\mathbb{S}^3 \backslash 4_1) - \frac{5\pi r}{12},  \Re\Phi^{(s)}_M\left(\frac{1}{3}-\frac{r}{2}\right)
, \Re\Phi^{(s)}_M \left(\frac{1}{3}+\frac{r}{2} \right) \right\}$$
for any $w \in \tilde C$. Similarly, for $k\geq 1$, if we consider the contour 
$$\tilde C' : \frac{5}{24} \to \frac{1}{3} - \frac{r}{2} \to \frac{1}{3} - \frac{r}{2} + \frac{r}{2}i \to  \frac{1}{3} + \frac{r}{2} + \frac{r}{2}i \to \frac{1}{3} + \frac{r}{2} \to \frac{11}{24},$$
then we have
$$  \int_C \exp\left( (N+1/2) \left( \Phi^{(s)}_M(w) + 2\pi i(k-1)w\right) \right) dw = \int_{\tilde C'} \exp\left( (N+1/2) \left( \Phi^{(s)}_M(w) + 2\pi i(k-1)w\right) \right) dw $$
with 
$$ \Re \left( \Phi^{(s)}_M(w) + 2\pi i(k-1)w\right) \leq \max\left\{ \frac{1}{2\pi} \Vol(\mathbb{S}^3 \backslash 4_1) - \frac{5\pi r}{12},  \Re\Phi^{(s)}_M\left(\frac{1}{3}-\frac{r}{2}\right)
, \Re\Phi^{(s)}_M \left(\frac{1}{3}+\frac{r}{2} \right) \right\}$$
for any $w \in \tilde C'$. This completes the proof of Lemma~\ref{mainlemma}.

\subsection{AEF for the Turaev-Viro invariants of the figure eight knot complement}

As an application of AEF's obtained in previous subsections, we are going to find out the AEF for the Turaev-Viro invariants of the figure eight knot complement as follows.

Recall from Theorem~\ref{relationship} that the TV invariants and the colored Jones polynomials of a link $L$ are related by
$$ TV_{r}\left(\mathbb{S}^3 \backslash L, e^{\frac{2\pi i}{r}}\right) = 2^{n-1}(\eta_{r}')^{2} \sum_{1\leq M \leq \frac{r-1}{2}}\left|\tilde{J}_{M}\left(L,e^{\frac{2\pi i }{N+\frac{1}{2}}}\right)\right|^2,$$
where $r=2N+1$ and $\ds \eta_{r}' = \frac{2\sin(\frac{2\pi}{r})}{\sqrt{r}}$.

For the figure eight knot $L = 4_1$, we can split the TV invariant into three parts.
\begin{align*}
&TV_{r}\left(\mathbb{S}^3 \backslash 4_1, e^{\frac{2\pi i}{r}}\right) \\
= &(\eta_{r}')^{2} \sum_{M: s \in (1-\zeta,1]}\left|\tilde{J}_{M}\left(L,e^{\frac{2\pi i }{N+\frac{1}{2}}}\right)\right|^2 + \sum_{M: s\in (\frac{1}{2}-\delta, \frac{1}{2}+\delta)}\left|\tilde{J}_{M}\left(L,e^{\frac{2\pi i }{N+\frac{1}{2}}}\right)\right|^2 \\
&+(\eta_{r}')^{2} \sum_{\substack{1\leq M \leq N,\\ s \notin (\frac{1}{2}-\delta, \frac{1}{2}+\delta)\\ s \notin (1-\zeta,1]}}\left|\tilde{J}_{M}\left(L,e^{\frac{2\pi i }{N+\frac{1}{2}}}\right)\right|^2
\end{align*}

The last summation can be estimated by using Lemma~\ref{mainlemma} and Lemma~\ref{uppbd} respectively. Using the same arguments as in the proof of Theorem~\ref{mainthm2}, we can see that the growth rate of the last term is strictly less than some multiple (a number in $(0,1)$) of that of the first and the second summations. As a result,
\begin{align*}
TV_{r}\left(\mathbb{S}^3 \backslash 4_1, e^{\frac{2\pi i}{r}}\right) 
\stackrel[N \to \infty]{\sim}{ }&(\eta_{r}')^{2}  \sum_{M: s \in (1-\zeta,1]}\left|\tilde{J}_{M}\left(L,e^{\frac{2\pi i }{N+\frac{1}{2}}}\right)\right|^2 
\end{align*}

Note that each summands satisfy the condition in the previous subsections. To apply the formulas obtained in previous section, we need the following lemma.

\begin{lemma}\label{asymchange}
For each $M \in \NN$, let $a^{M}_{N}$ and $b^{M}_{N}$ be two sequences of positive real numbers such that
$|a^{M}_{N} -  b^{M}_{N}| \leq b^{M}_N K(N)$, where $K(N)$ is a sequence of positive real numbers independent on $M$ such that $\ds \lim_{N \to \infty}K(N) = 0$. Then we have
$$ \sum_{M=1}^{N} a^{M}_{N}  \stackrel[N \to \infty]{\sim}{ } \sum_{M=1}^{N} b^{M}_{N}  .$$
\end{lemma}

In (\ref{asymch1}), the error term $O(\frac{\log N}{N})$ comes from Proposition~\ref{exprepn2} and Theorem~\ref{FSA}. From the proof of Proposition~\ref{exprepn2} and Theorem~\ref{FSA} we can see that the error depends continuously on the functions $\tilde\Phi_M(z)$. Since our functions $\tilde\Phi^{(s)}_M(z)$ converges uniformly to analytic functions $\tilde\Phi^{(s)}_0(z)$, the error terms can be controlled uniformly.

Next, since
\begin{align*}
\frac{e^{\frac{\pi i M}{N+1/2}}-e^{-\frac{\pi i M}{N+1/2}}}{e^{\frac{\pi i}{N+1/2}}-e^{-\frac{\pi i}{N+1/2}}} 
= \frac{\sin \left( \frac{\pi M}{N+1/2} \right)}{\sin \left( \frac{\pi}{N+1/2} \right)}
\stackrel[N \to \infty]{\sim}{ } \frac{N+1/2}{\pi} \sin(a\gamma)
\end{align*}
by (\ref{normal}), we have
\begin{align}
\tilde{J}_{M}\left(4_1,q \right)
&= \frac{e^{\frac{\pi i M}{N+1/2}}-e^{-\frac{\pi i M}{N+1/2}}}{e^{\frac{\pi i}{N+1/2}}-e^{-\frac{\pi i}{N+1/2}}} \times \tilde{J}_{M}\left(4_1,q \right) \notag \\
&\stackrel[M \to \infty]{\sim}{ } \frac{1}{\pi i}
 (N+\frac{1}{2})^{3/2} \frac{\sqrt{2\pi}\exp\left((N+\frac{1}{2})\tilde{\Phi}^{(s)}_{M}\left(z^{(s)}_{M}\right)\right)}{\sqrt{{\tilde{\Phi}}_{M}^{(s)''}(z^{(s)}_M)} }
\end{align}

Therefore, by Lemma~\ref{asymchange}, we have
\begin{align*}
&\quad\text{ } \sum_{M:s \in (1-\zeta,1]}\left|\tilde{J}_{M}\left(L,e^{\frac{2\pi i }{N+\frac{1}{2}}}\right)\right|^2 \\
&=\frac{2(N+1/2)^3}{\pi} \sum_{M :s \in (1-\zeta,1]} 
\left| \frac{\exp\left((2N+1)\tilde\Phi^{(s)}_{M}\left(z^{(s)}_{M}\right)\right)}{\tilde\Phi^{(s)''}_{M}(z^{(s)}_M) } \right|\left( 1 + O \left(\frac{\log N}{N} \right) \right)
\end{align*}

Note that the sum can be expressed in the form

\begin{align*}
&\quad \text{ }\sum_{M:s \in (1-\zeta,1]}\left|\tilde{J}_{M}\left(L,e^{\frac{2\pi i }{N+\frac{1}{2}}}\right)\right|^2 \\
&=\frac{2(N+1/2)^3}{\pi} \sum_{M :s \in (1-\zeta,1]} \left| \frac{\exp\left((2N+1)\tilde\Phi^{(s)}_{M}\left(z_{M}\right)\right)}{\tilde\Phi^{(s)''}_{M}(z_M) } \right|\left( 1 + O \left(\frac{\log N}{N} \right) \right)\\
&= \frac{2(N+1/2)^3}{\pi} \sum_{M :s \in (1-\zeta,1]} \left| \frac{\exp\left((2N+1)\Theta\left(\frac{M}{N+1/2}\right)\right)}{\Xi\left(\frac{M}{N+1/2}\right) } \right|  \left( 1 + O \left(\frac{\log N}{N} \right) \right),\\
\end{align*}

where the functions $\Theta(z)$ and $\Xi(s)$ are defined in previous subsections by
\begin{align*}
\Theta(s)
&= \tilde{\Phi}^{(s)}_{\infty}(z(s))\\
&=\frac{1}{2\pi i}\left[\operatorname{Li_{2}}\left(e^{-2\pi iz(s) + 2\pi i s}\right)- \operatorname{Li_{2}}\left(e^{ 2\pi iz(s) + 2\pi i s}\right)\right] +2 \pi i\left( 1 - s \right)z(s) \\
\Xi(s)
&= 2\pi i e^{2\pi i s} (e^{-2\pi i z(s)} - e^{2\pi i z(s)}).
\end{align*}

Next, we want to change the sum into integral. This can be done by the following proposition.
\begin{proposition}\label{Rietoint}
Let $f(z)$ be an analytic function defined on a domain $D$ containing $[a,b]$. Assume that
\begin{enumerate}
\item $x_{crit}\in [a,b]$ is the only critical point of $\Re f$ along $[a,b]$ on which $\Re f(z)$ attains its maximum;
\item $x_{crit}$ is non-degenerate with $(\text{Re}f)''(x_{crit})<0$.
\end{enumerate}
Then for any positive $C^1$ function $h(x)$ on $[a,b]$, we have the following asymptotic equivalence:
\begin{align*}
&\quad \text{ } \int_{a}^{b} h(x) \left| \exp \left( \left(N+\frac{1}{2} \right) f(x) \right) \right| dx \\
&=\sum_{\substack{k=1, \\\\ a + \frac{2k}{2N+1}\leq b}} \left(\frac{1}{N+1/2}\right) h\left(\frac{2k+1}{2N+1}\right) \left| \exp\left(  \left(N+\frac{1}{2}\right)  f \left(\frac{2k+1}{2N+1}\right)\right) \right| \\
&\qquad \qquad \times \left(1 + O\left(\frac{1}{(N+1/2)^{1/3}} \right) \right).
\end{align*}
\end{proposition}

By Proposition~\ref{Rietoint}, we have
\begin{align*}
\sum_{M :s \in (1-\zeta,1]} \left| \frac{\exp\left((2N+1)\Theta\left(\frac{M}{N+1/2}\right)\right)}{\Xi\left(\frac{M}{N+1/2}\right) } \right| 
\stackrel[N \to \infty]{\sim}{ } (N+\frac{1}{2})\int_{1-\zeta}^{1} \left| \frac{\exp\left((2N+1)\Theta(s)\right)}{\Xi(s) } \right| ds
\end{align*}

To find out the value of $\Re \Theta ''(1)$, recall that $z(s)$ satisfies the equation
\begin{align*}
\omega = \frac{(\beta^2 + 1 - \beta) - \sqrt{(\beta^2+1-3\beta)(\beta^2+1+\beta)}}{2\beta},
\end{align*}
where $\beta = e^{2\pi i s}$ and $\omega = e^{2\pi i z(s)}$. Note that $\beta(1)=1$, $\beta'(1) = 2\pi i$ and $\omega'(1) = 2\pi i z'(1) e^{5\pi i /3}$.
Differentiate both sides of the equation with respect to $s$ and put $s=1$, we can check that $z'(1)=0$. Furthermore, when $s=1$, we have
$$ z(1)= \frac{5}{6} \text{ and } \log(1-e^{-2\pi i z(s) + 2\pi i s})+ \log(1-e^{2\pi i z(s) + 2\pi i s}) =0$$
Therefore, from direct calculation, one can show that
$$ \Re \Theta'( 1 )  = 0 \text{ and } \Re \Theta'' (1) = -2\sqrt{3}\pi<0$$

As a result, by Laplace's method we have

\begin{align}\label{s1}
\sum_{M:s \in (1-\zeta,1]}\left|\tilde{J}_{M}\left(L,e^{\frac{2\pi i }{N+\frac{1}{2}}}\right)\right|^2 
&\stackrel[N \to \infty]{\sim}{ }  \frac{1}{2}\frac{2(N+1/2)^4}{\pi}  \frac{1}{\Xi(1)} \sqrt{\frac{2\pi}{(2N+1) |\Re\Theta'' (1)|}} \exp \left( (2N+1)\Theta(1)\right)\notag\\
&\hspace{8pt} = \hspace{8pt} \pi^{-1/2} \frac{1}{(2\pi\sqrt{3})^{3/2}} (N+\frac{1}{2})^{7/2}  \exp \left( (2N+1)\frac{\Vol(\mathbb{S}^3 \backslash 4_1)}{2\pi}\right)
\end{align}

Note that there is an extra $\frac{1}{2}$ in the above formula since the maximum point lies on the boundary. Overall, the AEF of the Tureav-Viro invariant of the figure eight knot is given by

\begin{align*}
TV_{r}\left(\mathbb{S}^3 \backslash 4_1, e^{\frac{2\pi i}{r}}\right)
&\stackrel[N \to \infty]{\sim}{ }  (\eta_{r}')^{2} \frac{r^{3/2}}{2} \frac{\pi^{-1/2}}{\sqrt{2}} \frac{1}{(2\pi\sqrt{3})^{3/2}} (\frac{r}{2})^2  \exp \left( \frac{r}{2\pi}\Vol(\mathbb{S}^3 \backslash 4_1)\right) \\
&\stackrel[N \to \infty]{\sim}{ }  (\frac{4\pi}{r^{3/2}})^{2} \frac{r^{3/2}}{2} \frac{\pi^{-1/2}}{\sqrt{2}} \frac{1}{(2\pi\sqrt{3})^{3/2}}(\frac{r}{2})^2  \exp \left(\frac{ r}{2\pi}\Vol(\mathbb{S}^3 \backslash 4_1)\right) \\
&\hspace{8pt}= \hspace{12pt} \left(\frac{1}{4}\right) \left(\frac{r}{2}\right)^{1/2} \left|\frac{2}{ \sqrt{-3}} \right|^{3/2} \exp \left(\frac{ r}{2\pi}\Vol(\mathbb{S}^3 \backslash 4_1)\right)
\end{align*}
This complete the proof of Theorem~\ref{mainthm4}.

\section{Proof of Results listed in Section~\ref{sec2}}\label{sec3}

\begin{proof}[Proof of Proposition~\ref{tanapp} and Proposition~\ref{tanapp2}]
We follow the line of the proof in \cite{WA17} and \cite{HM13} with suitable modification. First of all, recall that for $| \operatorname{Re}(z) | < \pi$, or $| \operatorname{Re}(z)|= \pi$ and $\operatorname{Im}(z)>0$,
\begin{align*}
\frac{1}{2i}\operatorname{Li}_{2}(-e^{iz})
&=\frac{1}{4}\int_{C_{R}}\frac{e^{zt}}{t^{2}\sinh(\pi t)}dt \\
\text{\hskip -6em which implies \hskip 4em}
S_{\gamma}(z)
&= \exp\left(\frac{1}{2i\gamma}\operatorname{Li}_{2}(-e^{iz}) + I_{\gamma}(z)\right) \\
&= \exp\left(\frac{M+a}{\xi}\operatorname{Li}_{2}(-e^{iz}) + I_{\gamma}(z)\right), \\
\text{\hskip -7em where \hskip 7.5em}
I_{\gamma}(z) &= \frac{1}{4}\int_{C_{R}}\frac{e^{zt}}{t\sinh(\pi t)} \left(\frac{1}{\sinh(\gamma t)}-\frac{1}{\gamma t}\right)dt\,.
\end{align*}

Recall that our function $g_{M}$ is given by
\begin{align*}
g_{M}(z) = \exp\left(-(M+a)\left(u-\frac{a\xi}{M+a}\right)z\right) \frac{S_{\gamma}(\pi - iu + i\xi z  + i\xi (\frac{a}{M+a}))}{S_{\gamma}(-\pi - iu - i \xi z + i\xi (\frac{a}{M+a}))}
\end{align*}

Substituting the above equation for $S_{\gamma}$ into the definition of $g_{M}$ leads to
\begin{align*}
\lefteqn{ g_{M}(z) = \exp\left[-(M+a)\left(u-\frac{a\xi}{M+a}\right)z \right] } \\
&\exp\left[\frac{M+a}{\xi}\left(\operatorname{Li}_{2}(e^{u- z\xi - \frac{a\xi}{M+a})})  - \operatorname{Li}_{2}(e^{u+z\xi - \frac{a\xi}{M+a}})  )\right)\right] \times\\
& \exp \left[ I_{\gamma}(\pi - iu + i\xi z + i\xi (\frac{a}{M+a})) - I_{\gamma}(-\pi - iu - i\xi z + i\xi (\frac{a}{M+a}))  \right]
\end{align*}
Let
$$\tilde{\Phi}^{(s)}_{M}(z)=\frac{1}{\xi}\left[\operatorname{Li_{2}}\left(e^{u-\left(z+\frac{a}{M+a}\right)\xi}\right) -
 \operatorname{Li_{2}}\left(e^{u + \left(z - \frac{a}{M+a}\right)\xi}\right)\right] - \left(u - \frac{a\xi}{M+a}\right)z $$
We have
\begin{align*}
g_{M}(z) =
& \exp\left[(M+a)\tilde{\Phi}^{(s)}_{M}(z)\right]\, \times \\
& \hspace*{16pt}  \exp \left[ I_{\gamma}(\pi - iu + i\xi z + i\xi (\frac{a}{M+a})) - I_{\gamma}(-\pi - iu - i\xi z + i\xi (\frac{a}{M+a}))  \right]
\end{align*}

Decompose $C_{+}(\epsilon)$ as $C_{+,1}, C_{+,2}$ and $C_{+,3}$ by $\ds \epsilon \to (\epsilon -\frac{u}{2\pi} + i) \to (1 -\epsilon -\frac{u}{2\pi} + i ) \to 1-\epsilon$ and
$C_{-}(\epsilon)$ as $C_{-,1}, C_{-,2}$ and $C_{-,3}$ by $\ds \epsilon \to (\epsilon + \frac{u}{2\pi} - i) \to (1 -\epsilon +\frac{u}{2\pi} - i ) \to 1 - \epsilon$.\\
Write $I_{\pm,i}(N)$ be the integral along $C_{\pm,i}$ respectively. We are going to show the following controls on the integrals:
\begin{align}
|I_{+,1}(N)| &< \frac{K_{+,1}}{M+a} \label{prop2.11}\\
|I_{+,2}(N)| &< \frac{K_{+,2}}{M+a} \label{prop2.12}\ \\
|I_{+,3}(N)| &< \frac{K_{+,3}}{M+a} \label{prop2.13}\\\
|I_{-,1}(N)| &< \frac{K_{-,1}}{M+a} \label{prop2.14}\ \\
|I_{-,2}(N)| &< \frac{K_{-,2}}{M+a} \label{prop2.15}\ \\
|I_{-,3}(N)| &< \frac{K_{-,3}}{M+a} \label{prop2.16}\
\end{align}

Let us observe the comparison between (i) $\tilde{\Phi}^{(s)}_{M}$, (ii) its limiting function $\tilde{\Phi}^{(s)}_\infty$ and (iii) the function $\Phi^{(1)}_\infty $ (which is $\Phi$ in \cite{HM13}).
\begin{align*}
\Phi^{(1)}_\infty (z) &= \frac{1}{\xi}(\operatorname{Li}_{2}(e^{u-\xi z}) - \operatorname{Li}_{2}(e^{u+\xi z})) - uz \\
\tilde{\Phi}_{M}^{(s)}(z) &= \frac{1}{\xi}\left[\operatorname{Li_{2}}\left(e^{u-\left(z+\frac{a}{M+a}\right)\xi}\right) -
 \operatorname{Li_{2}}\left(e^{u + \left(z - \frac{a}{M+a}\right)\xi}\right)\right] - \left(u - \frac{a\xi}{M+a}\right)z \\
\tilde{\Phi}^{(s)}_\infty (z) &= \frac{1}{\xi}\left(\operatorname{Li}_{2}(e^{u- z\xi - (1-s)\xi)})  - \operatorname{Li}_{2}(e^{u+z\xi - (1-s)\xi}) \right) -(u- (1-s))z
\end{align*}
The proof of the above estimates for the contour integrals is basically the same as the one of Proposition~3.1 in \cite{HM13}.

 To prove (\ref{prop2.11}), first we estimate $|\tan((M+a)\pi((-u/2\pi + i)t + \epsilon)) - i|$. By using (6.8) in \cite{HM13}, we have
$$ |\tan((M+a)\pi((-u/2\pi + i)t + \epsilon)) - i| \leq \frac{2e^{-2(M+a)\pi t}}{1-e^{-\pi^2 / u}} $$
So we have
\begin{align*}
&|I_{+,1}(N+n-2)|
\leq \frac{2}{1-e^{-\pi^2 / u}} \int_{0}^{1}e^{-2(M+a)\pi t} \left|g_{N}((-\frac{u}{2\pi} +i)t + \epsilon)\right| dt
\end{align*}
Recall the Lemma~6.1 in \cite{AH06} that for $|\operatorname{Re}(z)| \leq \pi$ we have
$$|I_{\gamma}(z)| \leq 2A + B|\gamma|\left(1+e^{-\operatorname{Im}(z)R}\right)$$
That means $ \exp \left[ I_{\gamma}(\pi - iu + i\xi z + i\xi (\frac{a}{M+a})) - I_{\gamma}(-\pi - iu - i\xi z + i\xi (\frac{a}{M+a}))  \right]$ is bounded above by some constant $K>0$ and
$$ \left|g_{N}((-\frac{u}{2\pi} +i)t + \epsilon)\right| \leq Ke^{(M+a) \operatorname{Re}\Phi_{N}((\frac{u}{2\pi}+i)t+\epsilon)} $$
From the proof of~(6.2) in~\cite{HM13}, we know that $\operatorname{Re}\Phi^{(1)}_{\infty}((\frac{u}{2\pi}+i)t+\epsilon)<0$ for sufficiently small~$\epsilon>0$. We have two cases:

\begin{enumerate}
\item if $a$ is fixed and $u \neq 0$, since we have
$$\tilde{\Phi}^{(s)}_M \xrightarrow{M \to \infty}  \Phi_{\infty}^{(1)},$$
for $M$ large enough we have $\operatorname{Re}\Phi^{(s)}_{M}((\frac{u}{2\pi}+i)t+\epsilon)<0$.
\item if $a = N-M+\frac{1}{2}$ and $u=0$, since we have
$$\tilde{\Phi}^{(s)}_M \xrightarrow{M \to \infty}\tilde{\Phi}^{(s)}_\infty \text{ and }\tilde{\Phi}^{(s)}_{\infty} \xrightarrow{s \to 1} \tilde{\Phi}_{\infty}^{(1)},$$
there exists a small $\zeta_1>0$ such that whenever $1-\zeta_1<s\leq 1$ and $M$ is large enough, we have $\operatorname{Re}\Phi^{(s)}_{M}((\frac{u}{2\pi}+i)t+\epsilon)\leq 0$.
\end{enumerate}

Hence we have
$$\left|I_{+,1}(N)\right| \leq \frac{2}{1-e^{-\pi^{2} / u}}K \int_{0}^{1} e^{-2(M+a)\pi t} dt \leq \frac{K_{+,1}}{M+a}$$
This establishes the inequality~(\ref{prop2.11}).
The proof of the other inequalities~(\ref{prop2.12}--\ref{prop2.16}) are basically the same.
\end{proof}

\begin{proof}[Proof of Proposition~\ref{exprepn} and Proposition~\ref{exprepn2}] Write
\begin{align*}
g_{M}(z)
&= \exp((M+a)\tilde{\Phi}^{(s)}_{M}(z)) \times \exp (\Delta(z)),
\end{align*}
where $\Delta (z) = I_{\gamma}(\pi - iu + i\xi z + i\xi (\frac{a}{M+a})) - I_{\gamma}(-\pi - iu - i\xi z + i\xi (\frac{a}{M+a})) $.
First, note that
\begin{align*}
&\quad| \int_{p(\epsilon)} g_{M}(\omega) d\omega - \int_{p(\epsilon)} \exp( (M+a) \tilde{\Phi}^{(s)}_{M}(\omega)) d\omega | \\
&= | \int_{p(\epsilon)} \exp((M+a)\tilde{\Phi}^{(s)}_{M}(\omega)) [ \exp(\Delta(\omega)) - 1] | d\omega\\
&\leq \max_{\omega \in p(\epsilon)} \{ \exp ((M+a) \operatorname{Re} \tilde{\Phi}^{(s)}_{M}(\omega) \} \int_{p(\epsilon)} | \exp(\Delta(\omega)) - 1| d\omega \\
&=  \max_{\omega \in p(\epsilon)} \{ \exp ((M+a) \operatorname{Re} \tilde{\Phi}^{(s)}_{M}(\omega) \} \int_{\epsilon}^{1-\epsilon} |h_{\gamma}(t)| dt
\end{align*}

where
\begin{align*}
h_{\gamma}(t)= &\sum_{n=1}^{\infty} \frac{1}{n!}[I_{\gamma}(\pi - iu + i\xi t + i\xi(\frac{a}{M+a}) ) -I_{\gamma}(-\pi - iu - i\xi t + i\xi(\frac{a}{M+a}))]^{n}
\end{align*}

In the above we use the analyticity of $h_{\gamma}(\omega)$ to change the contour to straight line parametrized by $t$, $t\in (\epsilon, 1-\epsilon)$.

Recall the lemma 3 in \cite{AH06} that there exist $A,B>0$ dependent only on $R$ such that if $|\operatorname{Re}(z)|<\pi$, we have
$$ |I_{\gamma}(z)| \leq A(\frac{1}{\pi-\operatorname{Re}(z)} + \frac{1}{\pi + \operatorname{Re}(z)})|\gamma| + B(1+e^{-\operatorname{Im}(z)R})|\gamma| $$

So we can find a positive constant $B'$ such that
\begin{align}\label{EEI1}
|I_{1}|&=|I_{\gamma}(\pi - iu + i\xi t + i\xi(\frac{a}{M+a})))|  \notag \\
&\leq A|\gamma| \left(  \frac{1}{2\pi(t+\frac{a}{M+a})} + \frac{1}{2\pi - 2\pi(t+\frac{a}{M+a})}\right) + B|\gamma|(1+e^{(u-u(t+\frac{a}{M+a}))R})  \notag\\
&\leq A|\gamma| \left(  \frac{1}{2\pi(t+\frac{a}{M+a})} + \frac{1}{2\pi - 2\pi(t+\frac{a}{M+a})}\right) + B'|\gamma|
\end{align}
and
\begin{align}\label{EEI2}
|I_{2}|&=|I_{\gamma}(-\pi - iu - i\xi t + i\xi(\frac{a}{M+a})|   \notag\\
&\leq A|\gamma| \left( \frac{1}{2\pi - 2\pi(t-\frac{a}{M+a})} + \frac{1}{2\pi(t-\frac{a}{M+a})}\right) + B|\gamma|(1+e^{(u-u(t-\frac{a}{M+a}))R})  \notag\\
&\leq A|\gamma|  \left( \frac{1}{2\pi - 2\pi(t-\frac{a}{M+a})} + \frac{1}{2\pi(t-\frac{a}{M+a})}\right) + B'|\gamma|
\end{align}

Let $f(t)=\dfrac{1}{t}+\dfrac{1}{1-t}$. Note that $f(t)\geq 4$ for $t \in (\epsilon,1-\epsilon)$.

From (\ref{EEI1}) and (\ref{EEI2}), we have
\begin{align*}
|I_{1}-I_{2}| \leq |\gamma| \left(\frac{A}{2\pi} \left( f \left(t+\frac{a}{M+a}\right) + f\left(t-\frac{a}{M+a}\right) \right) + B' \frac{f(t)}{4} \right)
\end{align*}

Note that for $t \in (\epsilon,1-\epsilon)$ and $N$ large,
\begin{align*}
\frac{f(t)} {f(t+\frac{a}{M+a})}
= \frac{t+\frac{a}{M+a}}{t} \times \frac{ 1 - (t+\frac{a}{M+a})}{1-t}
\geq  1 - \frac{\frac{a}{M+a}}{\epsilon}
= 1 - \frac{a}{2a+1} \geq \frac{1}{3}
\end{align*}

Similarly,
\begin{align*}
\frac{f(t)} {f(t-\frac{a}{M+a})}
= \frac{t-\frac{a}{M+a}}{t} \times \frac{ 1 - (t-\frac{a}{M+a})}{1-t}
\geq  1 - \frac{\frac{a}{M+a}}{\epsilon}
= 1 - \frac{a}{2a+1} \geq \frac{1}{3}
\end{align*}

Thus, there exists some positive constant $A'''$ such that
\begin{align*}
|I_{1}-I_{2}| \leq |\gamma|A''' f(t)
\end{align*}

As a result,
\begin{align*}
\int_{\epsilon}^{1-\epsilon} |h_{\gamma}(t)| dt
\leq \sum_{n=1}^{\infty} \frac{A''' |\gamma|^{n}}{n!} \int_{\epsilon}^{1-\epsilon} f(t)^n dt
\leq \sum_{n=1}^{\infty} \frac{A''' |\gamma|^{n}}{n!} \int_{|\gamma|}^{1-|\gamma|} f(t)^n dt
\end{align*}

Following the argument in \cite{AH06}, p.537, for $n \geq 1$ we have
$$\int_{|\gamma|}^{1-|\gamma|}  f(t)^{n}dt \leq 2^{2n+1} \int_{|\gamma|}^{1-|\gamma|}  \frac{dt}{t^{n}} $$

Also,
\begin{align*}
\int_{|\gamma|}^{1-|\gamma|}  \frac{dt}{t} &= \log(M+a) - \log(1) = \log(M+a) \text{ and } \\
\int_{|\gamma|}^{1-|\gamma|}  \frac{dt}{t^n} &= \frac{1}{n-1}\left(\frac{1}{|\gamma|^{n-1}} - 2^{n-1}\right) \leq \frac{1}{(n-1)|\gamma|^{n-1}} \text{ for $n \geq 2$}
\end{align*}

Therefore we have
\begin{align*}
\int_{\epsilon}^{1-\epsilon} |h_{\gamma}(t)| dt
&\leq  \sum_{n=1}^{\infty} \frac{1}{n!} (A''')^{n} |\gamma|^{n} \int_{|\gamma|}^{1-|\gamma|} f(t)^{n} dt \\
&\leq  2|\gamma| \left(4A''' \log(M+a) + \sum_{n=2}^{\infty} \frac{(4A''')^{n}}{(n-1)n!}\right) \\
&\leq \frac{|\xi|}{M+a} (4A''' \log(M+a) + e^{4A'''}- 4A''' -1)\\
&\leq \frac{K \log (M+a)}{M+a}
\end{align*}
\end{proof}
\pagebreak

\begin{proof}[Proof of Lemma~\ref{Sratio}]
We only prove the formula for $c = \frac{1}{2}$. The general case can be proved similarly. Note that

\begin{align*}
\frac{S_{\gamma}(-\pi - iu)}{S_{\gamma}(\pi - iu - 2\gamma)} 
&= \exp \left(\frac{1}{4} \int_{C_{R}} \frac{e^{-iut}e^{-\gamma t}}{\sinh (\pi t) \sinh (\gamma t)}(e^{-\pi t + \gamma t} - e^{\pi t - \gamma t}) \frac{dt}{t}\right) \\
&= \exp \left(\frac{1}{4} \int_{C_{R}}e^{-iut}e^{-\gamma t} \frac{\sinh(-\pi t + \gamma t)}{\sinh (\pi t) \sinh (\gamma t)} \frac{dt}{t}\right) \\
&= \exp \left(\frac{1}{2} \int_{C_{R}} \frac{e^{-iut}e^{-\gamma t}\coth(\pi t)}{t}  -  \frac{e^{-iut}e^{-\gamma t}\coth(\gamma t)}{t}  dt\right)
\end{align*}

Now we modify the proof in \cite{HM13}. For $r>0$, let $U_{i}, i = 1,2,3$ be the segments defined by $r \xrightarrow{U_{1}} r - r' i \xrightarrow{U_{2}}  -r  - r' i \xrightarrow{U_{3}}  -r$ with $r' = \frac{3\pi}{u}r$. Since the zeros of $\sinh(\pi t)$ and $\sinh(\gamma t)$  are discrete, for generic $r'$, $U_{2}$ does not pass through those singular points.

Now we want to show that for $i=1,2,3$,
\begin{align*}
\lim_{r \to \infty} \int_{U_{i}} \frac{e^{-iut}}{\sinh (\pi t) \sinh (\gamma t)}(e^{-\pi t} - e^{\pi t - 2\gamma t})\frac{dt}{t} = 0
\end{align*}

We will show the convergence on (i) $U_{1}$, (ii) $U_{3}$, (iii) $U_{2}$.\\

First of all we define 
\begin{align*}
r = \frac{(2l+\frac{3}{4})\pi}{4\pi^{2}/(N+\frac{1}{2})u + u/(N+\frac{1}{2})} \text{ where $l \in \NN$}
\end{align*}
Clearly $r \to \infty$ if and only if $l \to \infty$. The choice of $r$ helps us to avoid the pole of $\sinh(\gamma t)$ and get a good estimation of the integrals.
More precisely, for $s\in [0,r']$ we consider the four functions
\begin{align*}
p(s) = |1-e^{-2\pi(r-si)}|&,\quad q(s) = |e^{2\pi(r-si)} - 1 | \\
g(s) =  | e^{-2\gamma(r-si)} -1| \text{ and } k(s) &= |2\sinh(\gamma (r-si))|= |e^{\gamma(r-si)} - e^{-\gamma(r-si)}|
\end{align*}
In the above $g(s)$ is the distance between $e^{-2\gamma(r-si)}$ and $1$. These functions correspond to the terms appear in the integrals as shown later. Now we are going to construct lower bound for these functions. When $r$ is large,
\begin{align*}
p(s)&= |1-e^{-2\pi(r-si)}| \geq 1 - e^{-2\pi r} \geq 1/2 ;\\
q(s)&= |e^{2\pi(r-si)} - 1 | \geq e^{2\pi r} -1 \geq 1
\end{align*}
Also, one can check that $$g(s) =  | e^{-2\gamma(r-si)} -1| = | e^{R(s)}e^{i\theta(s)} - 1| ,$$
where $R(s)= \frac{us}{N+\frac{1}{2}} - \frac{2\pi r}{N+\frac{1}{2}}$ and $\theta(s)=\frac{ur}{N+\frac{1}{2}} + \frac{2\pi s}{N+\frac{1}{2}} $. Moreover, due to the choice of $r$,
\begin{itemize}
\item when $s = \frac{2\pi}{u}r$, we have $R(s)=0$, $\theta(s)=2l\pi + \frac{3\pi}{4}$;
\item when $s = \frac{2\pi}{u}r - \frac{N+\frac{1}{2}}{8}$, we have $R(s)=-\frac{u}{8}$, $\theta(s)=2l\pi +\frac{\pi}{2}$;
\item when $s = \frac{2\pi}{u}r + \frac{3(N+\frac{1}{2})}{8}$, we have $R(s)=\frac{3u}{8}$, $\theta(s)=2l\pi + \frac{3\pi}{2}$.
\end{itemize}

Since $R(s)$ and $\theta(s)$ are strictly increasing in $s$ and $g(s)$ is the distance between $ e^{R(s)}e^{i\theta(s)}$ and $1$,
\begin{itemize}
\item for $0 \leq s \leq \frac{2\pi}{u}r - \frac{N+\frac{1}{2}}{8}$, $\ds g(s) \geq \min_{|z|\leq e^{-u/8}} |z -1| = 1 - e^{-u/8} $
\item for $\frac{2\pi}{u}r - \frac{N+\frac{1}{2}}{4} \leq s \leq \frac{2\pi}{u}r + \frac{3(N+\frac{1}{2})}{8}$, since $\theta(s) \in [(2l+1)\pi + \frac{\pi}{2}, (2l+1)\pi + \frac{3\pi}{2}]$, we must have $g(s) \geq 1$.\\
\item for $\frac{2\pi}{u}r + \frac{3(N+\frac{1}{2})}{8} \leq s \leq  \frac{3\pi}{u}r$, $\ds g(s) \geq  \min_{|z|\geq e^{3u/8}} |z -1| = e^{3u/8} -1$.
\end{itemize}

Finally for $k(s)$, recall that $|\sinh(z)|^2 = \sinh^{2}(x) + \sin^{2}(y)$ for any $z = x + i y \in \CC$. Note that the function $$K(s)=|k(s)|^2=4 |\sinh(-\gamma (r-si))|^2 = 4\left|\sinh^2\left(\frac{R(s) + i \theta(s)}{2}\right)\right| = 4 \left(\sinh^2 \frac{R(s)}{2} + \sin^2 \frac{\theta(s)}{2}\right)$$ has derivative
$$K'(s) = 8(\sinh R(s) \frac{u}{N+1/2} + \sin \theta(s) \frac{2\pi}{N+1/2}) \geq 0$$
for $\theta(s) \in [2l\pi +\frac{\pi}{2}, (2l+1)\pi]$, $k(s)$ is increasing on $[2l\pi +\frac{\pi}{2}, (2l+1)\pi]$ with
$$k(s) \geq k(2l\pi +\frac{\pi}{2}) = |\sinh(-u/8)| = \sinh(u/8)$$

For $\theta(s) \notin  [2l\pi +\frac{\pi}{2}, (2l+1)\pi]$, we have
$$k(s) = |e^{\gamma(r-si)} - e^{-\gamma(r-si)}| \geq |  |e^{\gamma(r-si)}| - | e^{-\gamma(r-si)} | | \geq e^{u/8}-e^{-u/8}$$
The last inequality above is due to the fact that the function is strictly increasing on $R(s)$ and $R(s)>u/8$ when $\theta(s) \notin  [2l\pi +\frac{\pi}{2}, (2l+1)\pi]$.

To conclude, we can find positive constants $M_{1}$, $M_{2}$, $M_{3}$ and $M_{4}$  independent on $r$ such that
$$ \frac{1}{p(s)} \leq M_{1}, \quad \frac{1}{q(s)} \leq M_{2}, \quad \frac{1}{g(s)} \leq M_{3} \quad \text{and} \quad \frac{2}{k(s)}\leq M_{4}$$

Now we can get a good control of the integrals.\\
(i) On $U_{1}$,
\begin{align*}
\left|\int_{U_{1}} \frac{e^{-iut}}{\sinh (\pi t) \sinh (\gamma t)}(e^{-\pi t}) \frac{dt}{t})\right|
&\leq 4 \int_{0}^{r'} \left|\frac{e^{-iu(r-si)}}{r-si}\right| \left|\frac{e^{(-\pi ) (r-si)}}{\sinh(\pi (r-si))\sinh(\gamma (r-si))}\right| ds \\
&\leq \frac{4M_{2}M_{4}}{r} \int_{0}^{r'} e^{-us} ds\\
&= \frac{4M_{2}M_{4}}{ur}(1-e^{-ur'}) \xrightarrow{r \to \infty} 0.
\end{align*}

Similarly,
\begin{align*}
\left|\int_{U_{1}} \frac{e^{-iut}}{\sinh (\pi t) \sinh (\gamma t)}(e^{\pi t - 2\gamma t}) \frac{dt}{t})\right|
&\leq 4\int_{0}^{r'} \left|\frac{e^{-iu(r-si)}}{r-si}\right| \left|\frac{e^{(\pi - 2\gamma) (r-si)}}{\sinh(\pi (r-si))\sinh(\gamma (r-si))}\right| ds \\
&\leq \frac{4}{r} \int_{0}^{r'} e^{-us} |e^{-\gamma(r-si)}| \frac{1}{p(s)}\frac{1}{g(s)} ds\\
&\leq  \frac{4M_{1}M_{3}}{r} \int_{0}^{r'} e^{(-1 + \frac{1}{N+1/2})us - \frac{2\pi r}{N+1/2}} ds
\end{align*}

Hence
\begin{align*}
\left|\int_{U_{1}} \frac{e^{-iut}}{\sinh (\pi t) \sinh (\gamma t)}(e^{\pi t - 2\gamma t}) \frac{dt}{t})\right| 
&\leq  \frac{4M_{1}M_{3}}{r} \int_{0}^{r'} e^{(-1 + \frac{1}{N+1/2})us - \frac{2\pi r}{N+1/2}} ds\\
&\leq  \frac{4M_{1}M_{3}}{ur}(e^{(-1 + \frac{1}{N+1/2})ur' - \frac{2\pi r}{N+1/2}} - e^{ - \frac{2\pi r}{N+1/2}}) \\
& \xrightarrow{r \to \infty} 0,
\end{align*}

(ii) On $U_{3}$,
\begin{align*}
&\quad\left|\int_{U_{3}} \frac{e^{-iut}}{\sinh (\pi t) \sinh (\gamma t)}(e^{-\pi t }) \frac{dt}{t})\right|\\
&\leq 4 \int_{0}^{r'} \left|\frac{e^{-iu(-r-si)}}{-r-si}\right| \left|\frac{e^{-\pi (-r-si)}}{(e^{\pi(-r-si)}-e^{-\pi(-r-si)})(e^{\gamma(-r-si)}-e^{-\gamma(-r-si)})}\right| dt \\
&\leq \frac{4}{r} \int_{0}^{r'} e^{-us} \frac{1}{|e^{(2\pi+\gamma)(-r-si)} + e^{-\gamma{(-r-si)}} - e^{\gamma{(-r-si)}} - e^{(2\pi-\gamma)(-r-si)}|} dt
\end{align*}
Note that the modulus of the terms in the denominator are $e^{-2\pi r -\frac{us+2\pi r}{2N+1}}$, $e^{\frac{us+2\pi r}{2N+ 1}}$, $e^{-\frac{us+2\pi r}{2N+ 1}}$ and $e^{-2\pi r + \frac{us+2\pi r}{2N+1}}$ respectively. For large r, the dominant term is given by $e^{\frac{us+2\pi r}{2N+ 1}} \xrightarrow{r \to \infty} \infty$. This shows that the denominator is bounded below. So we can find some constant $M_{5}$ such that
\begin{align*}
\left|\int_{U_{3}} \frac{e^{-iut}}{\sinh (\pi t) \sinh (\gamma t)}(e^{-\pi t}) \frac{dt}{t})\right|
\leq\frac{M_{5}}{r} \int_{0}^{r'} e^{-us} dt
\leq \frac{M_{5}}{ur}(1-e^{-ur'}) \xrightarrow{r \to \infty} 0.
\end{align*}

Similarly,
\begin{align*}
&\quad\left|\int_{U_{3}} \frac{e^{-iut}}{\sinh (\pi t) \sinh (\gamma t)}(e^{\pi t - 2\gamma t}) \frac{dt}{t})\right|\\
&\leq 4 \int_{0}^{r'} \left|\frac{e^{-iu(-r-si)}}{-r-si}\right| \left|\frac{e^{(\pi - 2 \gamma)(-r-si)}}{(e^{\pi(-r-si)}-e^{-\pi(-r-si)})(e^{\gamma(-r-si)}-e^{-\gamma(-r-si)})}\right| dt \\
&\leq  \frac{4}{r} \int_{0}^{r'} e^{-us} \frac{1}{|e^{q_{1}} + e^{q_{2}} -e^{q_{3}} -e^{q_{4}}|}
\end{align*}
where $q_{1}$, $q_{2}$, $q_{3}$ and $q_{4}$ are given by
\begin{align*}
q_{1}&=3\gamma(-r-si), \quad \quad \quad \quad\quad \quad \hspace{2pt} q_{2}=(-2\pi+\gamma)(-r-si),\\
q_{3}&=-\gamma(-r-si), \quad \quad \quad \quad\quad \quad q_{4}=(-2\pi+3\gamma)(-r-si)
\end{align*}
Note that the modulus of the terms in the denominator are $e^{-\frac{3(us+2\pi r) }{2N+1}}$, $e^{2\pi r-\frac{us+2\pi r }{2N+1}}$, $e^{\frac{us+2\pi r }{2N+1}}$ and $e^{2\pi r-\frac{3(us+2\pi r) }{2N+1}}$ respectively. For large r, the dominant term is given by $e^{2\pi r-\frac{us+2\pi r }{2N+1}} \to \infty$. This shows that the denominator is bounded below. Again we can find some constant $M_{6}$ such that
\begin{align*}
\left|\int_{U_{3}} \frac{e^{-iut}}{\sinh (\pi t) \sinh (\gamma t)}(e^{\pi t - 2\gamma t}) \frac{dt}{t})\right|
&\leq \frac{M_{6}}{r} \int_{0}^{r'} e^{-us} dt \\
&\leq \frac{M_{6}}{ur}(1-e^{-ur'}) \xrightarrow{r \to \infty} 0.\\
\end{align*}

(iii) On $U_{2}$, we consider the expression
\begin{align*}
\frac{S_{\gamma}(-\pi + iu )}{S_{\gamma}(\pi - iu - 2\gamma)}
=  \exp \left(\frac{1}{2} \int_{C_{R}} \frac{e^{-iut}e^{-\gamma t}\coth(\pi t)}{t}  -  \frac{e^{-iut}e^{-\gamma t}\coth(\gamma t)}{t}  dt\right)
\end{align*}
Note that for $t=s - r'i$, $s \in [-r,r]$,
\begin{align*}
|e^{-\gamma t}|=e^{\frac{r' u - 2\pi s}{2N+1}} \leq e^{\frac{r'u + 2\pi r}{2N+ 1}} \leq e^{\frac{2u r'}{2N+1}}
\end{align*}
Write $\kappa=\alpha - \beta i$, where $\kappa = \pi$ or $\gamma$,
\begin{align*}
 \left|\int_{U_{2}} \frac{e^{-iut}\coth(\kappa t)}{t} e^{-\gamma t} dt \right|
\leq \int_{U_{2}} \left|\frac{e^{-iut}\coth(\kappa t)}{t}\right| |e^{-\gamma t}|dt 
\leq \frac{e^{-ur'(1-\frac{2 }{2N+1})}}{r'} \int_{U_{2}} \left|\coth(\kappa t)\right| dt
\end{align*}

By a similar trick as in \cite{HM13}, put $\ds \delta = \max_{-1 \leq s \leq 1} |\coth(\kappa s)| > 0$. This helps us to get away from the singularity of $\coth(s\pi)$ in the proof shown below. Now we have

\begin{align*}
\int_{U_{2}}|\coth(\kappa t)| dt 
&= \int_{-r}^{r}|\coth(s\alpha - r'\beta - (s\beta + \alpha r')i)| ds \\
&\leq  2\delta + \int_{-r}^{-1} | \frac{e^{s\alpha - r'\beta - (s\beta + \alpha r')i} + e^{-(s\alpha - r'\beta - (s\beta + \alpha r')i)}}{e^{s\alpha - r'\beta - (s\beta + \alpha r')i} - e^{-(s\alpha - r'\beta - (s\beta + \alpha r')i)}} | ds \\
&\quad \text{ }+ \int_{1}^{r} | \frac{e^{s\alpha - r'\beta - (s\beta + \alpha r')i} + e^{-(s\alpha - r'\beta - (s\beta + \alpha r')i)}}{e^{s\alpha - r'\beta - (s\beta + \alpha r')i} - e^{-(s\alpha - r'\beta - (s\beta + \alpha r')i)}} | ds \\
&\leq  2\delta + \int_{-r}^{-1}  \frac{|e^{s\alpha - r'\beta - (s\beta + \alpha r')i}| + |e^{-(s\alpha - r'\beta - (s\beta + \alpha r')i)}|}{|e^{s\alpha - r'\beta - (s\beta + \alpha r')i}| - |e^{-(s\alpha - r'\beta - (s\beta + \alpha r')i)}|}  ds \\
&\quad \text{ } + \int_{1}^{r}  \frac{|e^{s\alpha - r'\beta - (s\beta + \alpha r')i}| + |e^{-(s\alpha - r'\beta - (s\beta + \alpha r')i)}|}{|e^{s\alpha - r'\beta - (s\beta + \alpha r')i}| - |e^{-(s\alpha - r'\beta - (s\beta + \alpha r')i)}|}  ds\\
&= 2\delta +  \int_{-r}^{-1} \frac{|e^{s\alpha-r'\beta}| + |e^{-(s\alpha - r'\beta)}|}{|e^{s\alpha - r'\beta }| - |e^{-(s\alpha - r'\beta)}|}
+ \int_{1}^{r}  \frac{|e^{s\alpha-r'\beta}| + |e^{-(s\alpha - r'\beta)}|}{|e^{s\alpha - r'\beta }| - |e^{-(s\alpha - r'\beta)}|}  ds \\
&\leq  2\delta + \int_{1}^{r} \coth(s\alpha - r'\beta) ds +\int_{-r}^{-1} \coth(s\alpha - r'\beta) ds \\
&=  2\delta + \frac{\log(\sinh(\alpha r - r'\beta)) - \log(\sinh(\alpha - r'\beta))}{\alpha} \\
&\quad \text{ }+\frac{\log(\sinh(-\alpha - r'\beta)) - \log(\sinh(-\alpha r - r'\beta))}{\alpha}
\end{align*}
Hence
\begin{align*}
&\quad\left|\int_{U_{2}} \frac{e^{-iut}\coth(\kappa t)}{t} e^{-\gamma t} dt \right|\\
&\leq \frac{e^{-ur'(1-\frac{2 }{2N+1})}}{r'} [2\delta + \frac{\log(\sinh(\alpha r - r'\beta)) - \log(\sinh(\alpha - r'\beta))}{\alpha} \\
&\quad \text{ } +\frac{\log(\sinh(-\alpha - r'\beta)) - \log(\sinh(-\alpha r - r'\beta))}{\alpha}] \\
&\xrightarrow{r \to \infty} 0
\end{align*}

Let $C_{r}=[-r, -R] \cup \Omega_{R} \cup [R,r]$. Denote $U_{1}\cup U_{2} \cup U_{3}$ by $U_{123}$. By (i)-(iii) we get
\begingroup
\allowdisplaybreaks
\begin{align*}
&\quad \text{ }\int_{C_R} \frac{e^{-iut}e^{-\gamma t}}{\sinh (\pi t) \sinh (\gamma t)}(e^{-\pi t + \gamma t} - e^{\pi t - \gamma t}) \frac{dt}{t}) \\
&=\lim_{r \to \infty} \int_{C_{r}} \frac{e^{-iut}e^{-\gamma t}}{\sinh (\pi t) \sinh (\gamma t)}(e^{-\pi t + \gamma t} - e^{\pi t - \gamma t}) \frac{dt}{t})\\
&=\lim_{r \to \infty} \int_{C_{r}} (\frac{e^{-iut}\coth(\pi t)}{t} e^{-\gamma t} - \frac{e^{-iut}\coth(\gamma t)}{t} e^{-\gamma t}) dt \\
&=\lim_{r \to \infty} [\int_{U_{123} } (\frac{e^{-iut}\coth(\pi t)}{t} e^{-\gamma t} - \frac{e^{-iut}\coth(\gamma t)}{t} e^{-\gamma t} )dt \\
&\quad \text{ }-2\pi i\text{Res}(\frac{e^{-iut}\coth(\pi t)}{t}e^{-\gamma t}, t=-li) \\
&\qquad\qquad\qquad\qquad\qquad\qquad\qquad  + 2\pi i\text{Res}(\frac{e^{-iut}\coth(\gamma t)}{t}e^{-\gamma t}, t=\frac{-l\pi i}{\gamma}) ]\\
&=\lim_{r \to \infty} \int_{U_{123}} \frac{e^{-iut}e^{-\gamma t}}{\sinh (\pi t) \sinh (\gamma t)}(e^{-\pi t + \gamma t} - e^{\pi t - \gamma t}) \frac{dt}{t}) \\
&\quad \text{ }-2\pi i \sum_{l=0}^{\infty}\left[\text{Res}(\frac{e^{-iut}\coth(\pi t)}{t} e^{-\gamma t} ,t=-li) \right. \\
&\qquad\qquad\qquad\qquad\qquad\qquad\qquad \left. - \text{Res}(\frac{e^{-iut}\coth(\gamma t)}{t} e^{-\gamma t},t=\frac{-l\pi i}{\gamma})\right]\\
&=-2\pi i \sum_{l=0}^{\infty}\left[\text{Res}(\frac{e^{-iut}\coth(\pi t)}{t} e^{-\gamma t} ,t=-li) \right. \\
&\qquad\qquad\qquad\qquad\qquad\qquad\qquad \left. - \text{Res}(\frac{e^{-iut}\coth(\gamma t)}{t} e^{-\gamma t},t=\frac{-l\pi i}{\gamma})\right]
\end{align*}
\endgroup

In the above, the negative sign before the residue term is due to the negative orientation of the contour. Moreover, the term $l=0$ correspond to the residue at zero. To find out the residue at $z=0$, we consider the following series expansions:
$$ e^{-iuz} = 1 - iuz + \frac{(-iuz)^2}{2} + \dots, \quad \coth(\kappa z) = \frac{1}{\kappa z} + \frac{\kappa z}{3} + \frac{(\kappa z)^3}{45}  + \dots$$
$$ e^{-\gamma z}=1 -\gamma z + \frac{(-\gamma z)^2}{2} + \dots $$

Hence the residue (coefficient of $z^{-1}$) is given by $\ds \frac{-iu-\gamma}{\kappa}$ where $\kappa = \pi$ or $\gamma$. From this we can find that
\begin{align*}
&\quad \text{ }  \int_{C_{r}} \frac{e^{-iut}\coth(\pi t)}{t} e^{-\gamma t}\\
&= -2\pi i \sum_{l=0}^{\infty}\text{Res}(\frac{e^{-iut}\coth(\pi t)}{t} e^{-\gamma t} ,t=-li)\\
&=  -2\pi i \left(\frac{-iu-\gamma}{\pi} + \sum_{l=1}^{\infty}\text{Res}(\frac{e^{-iut}\coth(\pi t)}{t} e^{-\gamma t} ,t=-li) \right)\\
&=  -2u + 2\gamma i + 2\pi i \sum_{l=1}^{\infty}\frac{e^{l(-u+\gamma i)}}{l\pi i}\\
&=  -2u + 2\gamma i - 2\log(1-e^{-u+\gamma i})
\end{align*}

Similarly,
\begin{align*}
&\quad \text{ } \int_{C_{r}} \frac{e^{-iut}\coth(\gamma t)}{t} e^{-\gamma t} \\
&= -2\pi i \sum_{l=0}^{\infty}\text{Res}(\frac{e^{-iut}\coth(\gamma t)}{t} e^{-\gamma t} ,t=\frac{-l\pi i}{\gamma})\\
&=  -2\pi i \left(\frac{-iu-\gamma}{\gamma} + \sum_{l=1}^{\infty}\text{Res}(\frac{e^{-iut}\coth(\pi t)}{t} e^{-\gamma t} ,t=\frac{-l\pi i}{\gamma}) \right)\\
&=  -\frac{2u\pi}{\gamma} + 2\pi i + 2\pi i \sum_{l=1}^{\infty}\frac{e^{l(-\frac{u\pi}{\gamma} +\pi i)}}{l\pi i}\\
&=  -\frac{2u\pi}{\gamma} + 2\pi i - 2\log(1+e^{-\frac{u\pi}{\gamma} })
\end{align*}

Overall we have
\begin{align*}
\frac{S_{\gamma}(-\pi-iu)}{S_{\gamma}(\pi-iu-2\gamma)}
&= \frac{e^{u\pi/\gamma -\pi i} (1+e^{-u\pi / \gamma})}{e^{u-\gamma i}(1-e^{-u+\gamma i})}
=- \frac{e^{u\pi / \gamma}+1}{e^{u-\gamma i}-1}
\stackrel[N \to \infty]{\sim}{ }
-\frac{e^{2\pi i u (N+\frac{1}{2}) / \xi}}{e^{u}-1}
\end{align*}

In particular, when $u=0$,
\begin{align*}
\frac{S_{\gamma}(-\pi)}{S_{\gamma}(\pi-2\gamma)}
=\lim_{u \to 0}\frac{S_{\gamma}(-\pi-iu)}{S_{\gamma}(\pi-iu-2\gamma)}
= \frac{2}{1-e^{-\gamma i}}
\stackrel[N \to \infty]{\sim}{ }
\frac{2N+1}{\pi i}
\end{align*}

\end{proof}

\begin{proof}[Proof of Proposition~\ref{positive} and Proposition~\ref{positive2}]
From Lemma 3.5 in \cite{HM13} we know that $\operatorname{Re}\Phi^{(1)}_\infty(z^{(1)}_{\infty})>0$ for $0\leq u<\log((3+\sqrt{5})/2)$. Since $\Phi^{(s)}_{M}(z^{(s)}_{M}) \to \Phi^{(1)}_\infty(z^{(1)}_{\infty})$ as $M \to \infty$ and $s \to 1$, we get the result.
\end{proof}

\begin{proof}[Proof of Lemma~\ref{diff1}]
Recall that
\begin{align*}
\Phi_{\infty} (z) =&\frac{1}{\xi}(\operatorname{Li}_{2}(e^{u-\xi z}) - \operatorname{Li}_{2}(e^{u+\xi z})) - uz \\
\Phi_{M}(z)=&\frac{1}{\xi}\left[\operatorname{Li_{2}}\left(e^{u-\left(z+\frac{a}{M+a}\right)\xi}\right) -
 \operatorname{Li_{2}}\left(e^{u + \left(z - \frac{a}{M+a}\right)\xi}\right)\right] - uz
\end{align*}
Put $\ds y = \frac{a}{M+a}$, we have
\begin{align*}
\Phi_{M}(z) - \Phi_\infty (z) =  \frac{1}{\xi} \left[\left(\operatorname{Li}_{2}(e^{u-\xi z -\xi y)}) -\operatorname{Li}_{2}(e^{u-\xi z})\right) - \left( \operatorname{Li}_{2}(e^{u+\xi z - \xi y}) - \operatorname{Li}_{2}(e^{u+\xi z}) \right)\right]
\end{align*}
As a result, by L'Hospital's rule
\begin{align*}
&\quad \lim_{M \to \infty} (M+a)(\Phi_{M}(z) - \Phi_\infty(z)) \\
&= \left( \frac{a}{\xi} \lim_{y \to 0} \frac{ \left[\left(\operatorname{Li}_{2}(e^{u-\xi z -\xi y)}) -\operatorname{Li}_{2}(e^{u-\xi z})\right) - \left( \operatorname{Li}_{2}(e^{u+\xi z - \xi y}) - \operatorname{Li}_{2}(e^{u+\xi z}) \right)\right]}{y} \right)  \\
&= \frac{a}{\xi} \lim_{y \to 0} \frac{d}{dy}\left[\left(\operatorname{Li}_{2}(e^{u-\xi z -\xi y)}) -\operatorname{Li}_{2}(e^{u-\xi z})\right) - \left( \operatorname{Li}_{2}(e^{u+\xi z - \xi y}) - \operatorname{Li}_{2}(e^{u+\xi z}) \right)\right]\\
&= a[\log(1-e^{u-z\xi}) - \log(1-e^{u+z\xi})]
\end{align*}
\end{proof}

\begin{proof}[Proof of Lemma~\ref{diff2}]
To remove the $M$ dependence of $z_{M}$, recall that from (\ref{quadeqn}) that
\begin{equation*}
 AB\omega_M^2 - (A^2 + B^2- AB^2 )\omega_M +AB =0
\end{equation*}
where $\ds A=e^{u}$, $\ds B= e^{\frac{a\xi}{M+a}}$ and $\ds \omega_{M} = e^{z_{M}\xi}$.
When $b=1$, we have the equation
$$  A\omega_\infty^2 - (A^2 -A + 1)\omega_\infty +A =0 $$
By subtracting two equations we get
\begin{align*}
&\quad A(\omega_M^2-\omega_\infty^2)-(A^2 -A +1)(\omega_M - \omega_\infty)  \\
&= -A(B^2-1)\omega_M^2 + (A^2(B-1) + (B+1)(B-1))  -A(B-1)\\
\end{align*}
This implies
\begin{align*}
 \omega_{M} -\omega_\infty
=& (B-1)\frac{-A(B+1)\omega_M^2 + (A^2+B+1) - A}{A(\omega_{M} + \omega_\infty) - (A^2 - A +1) }
\end{align*}
For simplicity, we denote the right hand side by $(B-1)K_{M}$. Note that $ K_{M} \xrightarrow{M \to \infty} K \neq 0$.
On the other hand,
\begin{align*}
\omega_{M}  -\omega_\infty
&=  e^{z_{M} \xi} - e^{z_\infty \xi} \\
&=  e^{z_\infty \xi} ( e^{(z_{M} - z_\infty) \xi} - 1 ) \\
&=  e^{z_\infty \xi} (z_{M} - z_\infty) \xi \left( \sum_{k=1}^{\infty} \frac{((z_{M} - z_\infty) \xi)^{k-1}}{k!}\right)
\end{align*}

As a result,
\begin{align*}
 z_{M} - z_\infty
&= (B-1)\frac{K_{M}}{\xi e^{z_\infty \xi}( \sum_{k=1}^{\infty} \frac{((z_{M} - z_\infty) \xi)^{k-1}}{k!})} \\
&=  \frac{a\xi}{M+a} \left(\sum_{k=1}^{\infty} \frac{[a\xi /(M+a)]^{k-1} }{k!}\right) \frac{K_{M}}{ \xi e^{z_\infty \xi}( \sum_{k=1}^{\infty} \frac{((z_{M} - z_\infty) \xi)^{k-1}}{k!})}\\
&=  \frac{L_{M}}{M+a},
\end{align*}
where $L_{M} \xrightarrow{M \to \infty} L < \infty$. Therefore, we have
\begin{align*}
\lim_{M \to \infty} (M+a)(\Phi_\infty(z_{M}) - \Phi_\infty(z_\infty))
=\lim_{M \to \infty} \frac{\Phi_\infty(z_\infty + \frac{L_{M}}{M+a}) - \Phi_\infty(z_\infty)}{\frac{L_{M}}{M+a}} L_{M}
=  0,
\end{align*}
where in the last equality we use the fact that $z_\infty$ is the solution of the saddle point equation 
$$\dfrac{d\Phi_\infty(z)}{dz} = 0$$
\end{proof}

\begin{proof}[Proof of Proposition~\ref{Rietoint}]
First of all, note that the error term $E$ can be expressed in the form
\begin{align*}
E&= \int_{a}^{b} h(x) \left| \exp \left( \left(N+\frac{1}{2}\right) f(x) \right) \right| dx \\
&\quad\quad\quad - \sum_{\substack{k = 1,\\\\ a \leq \frac{2k+1}{2N+1} \leq b}} \frac{1}{N+1/2} h\left(\frac{2k+1}{2N+1}\right) \left| \exp\left( \left(N + \frac{1}{2} \right) f \left(\frac{2k+1}{2N+1}\right) \right) \right| \\
&= \sum_{\substack{k = 1,\\\\ a \leq \frac{2k+1}{2N+1} \leq b}} E(k),
\end{align*}
where $E(k)$ is defined by
\begin{align*}
E(k)
&= \int_{a+\frac{2k-1}{2N+1}}^{a+\frac{2k+1}{2N+1}} \left(h(x) \left| \exp \left( \left(N+\frac{1}{2}\right) f(x) \right) \right|\right.\\
&  \left.\quad\quad\quad -h\left(\frac{2k+1}{2N+1}\right) \left| \exp\left( \left(N + \frac{1}{2} \right) f \left(\frac{2k+1}{2N+1}\right) \right) \right|  \right)dx
\end{align*}
For each $N$, let $a_N$ and $b_N$ be the least and largest integers such that
$$x_{crit} - \frac{1}{\left(N+1/2 \right)^{1/3}} \leq  a + \frac{2a_N+1}{2N+1} < a + \frac{2b_N+1}{2N+1} \leq x_{crit} + \frac{1}{\left(N+ 1/2\right)^{1/3}}$$
Note that $$ \left(a + \frac{2a_N+1}{2N+1}\right) -  \left(a + \frac{2b_N+1}{2N+1}\right) \leq \frac{2}{\left(N+1/2\right)^{1/3}}$$
Then the error term can be splited into two parts:
\begin{align*}
E= &\sum_{k = a_N}^{b_N} E(k)  +  \sum_{k < a_N} E(k)  +\sum_{k > b_N} E(k)
\end{align*}

By Laplace's method, we know that
\begin{align*}
\int_{a}^{b} h(x) \left| \exp \left( \left(N+\frac{1}{2}\right) f(x) \right) \right|dx
&\hspace{8pt}=\hspace{8pt} \int_{a}^{b}  h(x)\exp \left(\left(N+\frac{1}{2}\right) \text{Re}f(x)\right) dx \\
& \stackrel[N \to \infty]{\sim}{ } \frac{\sqrt{2\pi}h(x_{crit}) \exp((N+1/2)\text{Re}f(x_{crit}))}{\sqrt{\left(N+1/2\right)}\sqrt{-\text{Re}f''(x_{crit})}}
\end{align*}
For the second and the third sum of $E$, note that we have
$$ \left|x_{crit} - \frac{2k+1}{2N+1}\right| > \frac{1}{(N+1/2)^{1/3}}.$$
So for each such $k$,
\begin{align*}
& \left| \exp\left( \left(N+\frac{1}{2}\right) f \left(\frac{2k+1}{2N+1}\right) \right) \right| / \left| \exp \left( \left(N+\frac{1}{2}\right) f(x_{crit})\right) \right|  \\
= & \left| \exp \left(\left(N+\frac{1}{2}\right) (f(x_{crit}) + \frac{f''(x_{crit})}{2}(x_{crit} - \frac{2k+1}{2N+1})^2 + \dots) - \left(N+\frac{1}{2}\right) f(x_{crit})\right) \right| \\
\leq & \left| \exp \left( \left(N+\frac{1}{2}\right)^{1/3}\frac{f''(x_{crit})}{2} + O(1)\right) \right| \xrightarrow{N \to \infty} 0
\end{align*}

Therefore, the second and third sums decay exponentially when they are compared with the integral $\ds \int_{a}^{b} h(x) \left| \exp \left( \left(N+\frac{1}{2}\right) f(x) \right) \right|dx $.

For the first sum, by the Mean-Value Theorem, for each $i$ there exists some 
$$\ds \xi _k \in \left(a + \frac{2k-1}{2N+1}, a + \frac{2k+1}{2N+1}\right)$$
scuh that
\begin{align*}
&\quad \left| \left| \exp \left( \left(N+\frac{1}{2}\right) f(x) \right) \right| - \left| \exp \left( \left(N+\frac{1}{2}\right) f \left(\frac{2k+1}{2N+1} \right) \right) \right| \right| \\
&\leq \left| \left(N+\frac{1}{2}\right)(\text{Re}f)'(\xi_i) \right| \exp\left(\left(N+\frac{1}{2}\right) \Re f(\xi_i)\right) \left( \frac{1}{N+1/2} \right) \\
&= \left| (\text{Re}f)'(\xi_i) \right| \exp\left(\left(N+\frac{1}{2}\right) \Re f(\xi_i)\right)
\end{align*}

For $\ds x \in \left[\frac{2a_N+1}{2N+1},\frac{2b_N+1}{2N+1}\right]$, define a function $g(x)$ by
$$g(x) = \left| (\text{Re}f)'(x) \right| \exp\left(\left(N+\frac{1}{2}\right) \Re f(x)\right).$$
Note that
$g(x_{crit}) = 0$. By assumption, since $x_{crit}$ is the only critical point for $\Re f$ and $\Re f$ attains its maximum at $x_{crit}$, we have
$$g(x) =
\begin{cases}
(\text{Re}f)'(x) \exp\left(\left(N+\frac{1}{2}\right) \Re f(x)\right) & \text{if $x\in [\frac{2a_N+1}{2N+1},x_{crit}]$}\\
 - (\text{Re}f)'(x) \exp\left(\left(N+\frac{1}{2}\right) \Re f(x)\right) & \text{if $x\in [x_{crit},\frac{2b_N+1}{2N+1}]$}
\end{cases}$$

For $x\in [\frac{2a_N+1}{2N+1},x_{crit}]$, $g'(x)$ is given by
\begin{align*}
g'(x)
&=\left(N+\frac{1}{2}\right) \left((\text{Re}f)'(x)\right)^2 \exp\left(\left(N+\frac{1}{2}\right)\text{Re}f(x) \right) \\
&\quad + (\text{Re}f)''(x) \exp\left(\left(N+\frac{1}{2}\right)\text{Re}f(x)\right)
\end{align*}
Let $x_{max}$ be the maximum point of $g(x)$ on $[\frac{a_N}{N},x_{crit}]$. Note that $g'(x_{crit})<0$. We have two possibilities:
(1) $x_{max} \in  (\frac{a_N}{N},x_{crit})$ and (2) $x_{max} \notin  (\frac{a_N}{N},x_{crit})$.

\begin{enumerate}
\item[(1)]
if $x_{max} \in  (\frac{a_N}{N},x_{crit})$, we have $g'(x_{max})=0$, i.e.
$$\ds (\text{Re}f)'(x_{max}) = \sqrt{- \left(N+\frac{1}{2}\right)^{-1}(\text{Re}f)''(x_{max}) }$$
Thus,
\begin{align*}
g(x_{max})
&=\sqrt{- \left(N+\frac{1}{2}\right)^{-1}(\text{Re}f)''(x_{max}) } \exp(N\text{Re}f(x_{max}))\\
&\leq \sqrt{- \left(N+\frac{1}{2}\right)^{-1}(\text{Re}f)''(x_{max}) } \exp(N\text{Re}f(x_{crit}))
\end{align*}
and
\begin{align*}
&\quad \left| \left| \exp \left( \left(N+\frac{1}{2}\right) f(x) \right) \right| - \left| \exp \left( \left(N+\frac{1}{2}\right) f \left(\frac{2k+1}{2N+1} \right) \right) \right| \right| \\
&\leq \sqrt{- \left(N+\frac{1}{2}\right)^{-1}(\text{Re}f)''(x_{max}) } \exp(N\text{Re}f(x_{crit}))
\end{align*}

On the other hand, by Mean Value Theorem, for each $k$ there exists some 
$$ d _k \in \left(a + \frac{2k-1}{2N+1}, a + \frac{2k+1}{2N+1}\right)$$
such that
$$  \left|h(x) - h\left(\frac{2k+1}{2N+1}\right)\right| = | h'(d_i) | \left|x-\frac{2k+1}{2N+1}\right|$$
Let $y_1$ and $y_2$ be the maximum points of $h(x)$ and $h'(x)$ on $[a,b]$ respectively, and let $c_N$ be the largest integer such that $a + \frac{2c_N+1}{2N+1} \leq x_{crit}$. Altogether,

\begin{align*}
 \sum_{k=a_N }^{c_N} E(k) 
&\leq  \sum_{k=a_N }^{c_N}  \int_{a + \frac{2k-1}{2N+1}}^{a + \frac{2k+1}{2N+1}}  h\left (\frac{2k+1}{2N+1}\right) \\
&\qquad \times \left(\left| | \exp (\left(N+\frac{1}{2}\right) f(x))| - | \exp\left(\left(N+\frac{1}{2}\right) f \left(\frac{2k+1}{2N+1}\right)\right) \right|\right) dx \\
&\quad+ \sum_{k=a_N }^{c_N}  \int_{a + \frac{2k-1}{2N+1}}^{a + \frac{2k+1}{2N+1}}  \int_{\frac{i-1}{N}}^{\frac{i}{N}} \left| h(x)- h\left (\frac{2k+1}{2N+1}\right)\right| \left| \exp (\left(N+\frac{1}{2}\right) f(x))\right| dx\\
&\leq  h(y_1) \sum_{k=a_N }^{c_N}  \int_{a + \frac{2k-1}{2N+1}}^{a + \frac{2k+1}{2N+1}} \sqrt{- \left(N+\frac{1}{2}\right)^{-1}(\text{Re}f)''(x_{max}) } \\
&\qquad \times \exp\left(\left(N+\frac{1}{2}\right)\text{Re}f(x_{crit})\right) dx \\
&\quad+ h'(y_2) \sum_{k=a_N }^{c_N}  \int_{a + \frac{2k-1}{2N+1}}^{a + \frac{2k+1}{2N+1}}  \frac{1}{N+1/2} \left| \exp \left(\left(N+\frac{1}{2}\right) f(x)\right)\right| dx\\
&\leq  h(y_1)\sqrt{-\text{Re}f''(x_{max}) } \exp\left(\left(N+\frac{1}{2}\right)\text{Re}f(x_{crit})\right)  \left( \frac{2}{(N+1/2)^{5/6}}\right)\\
&\quad+ \frac{ h'(y_2) }{N+1/2} \sum_{k=a_N }^{c_N}  \int_{a + \frac{2k-1}{2N+1}}^{a + \frac{2k+1}{2N+1}}  \left| \exp \left(\left(N+\frac{1}{2}\right) f(x)\right)\right| dx\\
&\leq  \sqrt{-\text{Re}f''(x_{max}) }\exp\left(\left(N+\frac{1}{2}\right)\text{Re}f(x_{crit})\right)\left( \frac{2}{(N+1/2)^{5/6}}\right)\\
&\quad+ \frac{ h'(y_2) }{N+1/2}  \int_{a}^{b}  \left| \exp \left(\left(N+\frac{1}{2}\right) f(x)\right)\right| dx
\end{align*}

Since
\begin{align*}
\frac{\sqrt{-\text{Re}f''(x_{max}) }\exp\left(\left(N+\frac{1}{2}\right)\text{Re}f(x_{crit})\right)\left( \frac{2}{(N+1/2)^{5/6}}\right)
}{ \frac{\sqrt{2\pi}h(x_{crit}) \exp((N+1/2)\text{Re}f(x_{crit}))}{\sqrt{\left(N+1/2\right)}\sqrt{-\text{Re}f''(x_{crit})}}} \xrightarrow{N \to \infty} 0
\end{align*}
and
\begin{align*}
\frac{\frac{ h'(y_2) }{N+1/2}  \int_{a}^{b}  \left| \exp \left(\left(N+\frac{1}{2}\right) f(x)\right)\right| dx}{\frac{\sqrt{2\pi}h(x_{crit}) \exp((N+1/2)\text{Re}f(x_{crit}))}{\sqrt{\left(N+1/2\right)}\sqrt{-\text{Re}f''(x_{crit})}}} \xrightarrow{N \to \infty} 0,
\end{align*}
we have the desired result.\\

\item[(2)]
if $x_{max} \notin  (\frac{a_N}{N},x_{crit})$, since $g'(x_{crit}) < 0$, we have $x_{max} = a + \frac{2a_N+1}{2N+1}$. 

In particular, since $|x_{crit} - x_{max}| \geq \dfrac{1}{2N^{1/3}}$, we have
\begin{align*}
&\quad  g(x_{max}) / | \exp (N f(x_{crit}))|  \\
&\leq  |  N(\text{Re}f)'(x_{max})| \left| \exp \left( \left(N+\frac{1}{2}\right)^{1/3}\frac{f''(x_{crit})}{2} + O(1)\right) \right|\\
& \xrightarrow{N \to \infty} 0
\end{align*}
\end{enumerate}
Hence in this case the error term decay exponentially compared with the integral 
$$\ds \int_{a}^{b} \left| \exp \left( \left(N+ \frac{1}{2} \right) f(x)\right)\right| dx.$$

Similar method can be applied to the interval $\left[x_{crit}, a + \frac{2b_N}{2N+1}\right]$. This completes the proof.
\end{proof}

\begin{proof}[Proof of Lemma~\ref{lemmaFSA}]
We are going to construct the contour using the same idea as in the proof of Lemma 3.4 of \cite{HM13}. To do so, we only need to check that the conditions in the construction are also satisfied in our case.\\

Let $q_{M}(t)=z^{(s)}_{M}t$ for $0<t<\operatorname{Re}(1/z^{(s)}_{M})$. Since $\ds \lim_{s \to 1} \lim_{M \to \infty}z^{(s)}_{M}=z^{(1)}_\infty<1$ (see the proof of Lemma 3.4 of \cite{HM13}), $\operatorname{Re}(1/z^{(s)}_{M})>1$ for $s$ sufficiently close to 1 and $M$ sufficiently large.

Also, since $d^{2}\Phi^{1}_\infty (z^{(1)}_\infty)/dz^{2} \neq 0$ and $\ds \lim_{s \to 1} \lim_{M \to \infty} \Phi^{(s)}_{M} \to \Phi^{(1)}_\infty$, we have $$d^{2}\Phi^{(s)}_{M}(z^{(s)}_{M})/d z^{2} \neq 0$$
for $s$ sufficently close to 1 and $M$ sufficiently large.

By definition we have $d \Phi^{(s)}_{M}(z^{(s)}_{M}) /dz =0$. This implies
$$\operatorname{Re} d\Phi^{(s)}_{M}(q_{M}(1)) /dt=0$$
for any $M$. Since $\max \{\operatorname{Re}\Phi^{(1)}_\infty(z)\}$ takes place at $z=z^{(1)}_\infty$, we must have
$$\max \{\operatorname{Re}\Phi^{(s)}_{M}(z)\} = \operatorname{Re}\Phi^{(s)}_{M}(z^{(s)}_{M})$$
along the line $q_{M}(t)$ for $s$ sufficiently close to 1 and $M$ sufficiently large.\\

Moreover, from the proof of Lemma 3.4 of \cite{HM13} that the difference between the argument of $z^{(1)}_\infty$ and $1/\sqrt{-d^{2}\Phi^{(1)}_\infty(z^{(1)}_\infty)/dz^{2}}$ is strictly smaller than $\pi/4$. Hence the difference between the argument of $z^{(s)}_{M}$ and $1/\sqrt{-d^{2}\Phi^{(s)}_{M}(z^{(s)}_{M})/dz^{2}}$ is also strictly smaller than $\pi/4$ for $s$ sufficiently close to 1 and $M$ sufficiently large. As a result the same construction of the path $Q$ in the proof of Lemma 3.4 of \cite{HM13} still applies. \\

Finally we connect $z^{(s)}_{M}(\operatorname{Re}1/z^{(s)}_{M})$ and $1$ by a line segment $L$. Since from the proof of Lemma 3.4 in \cite{HM13} that $\operatorname{Re}\Phi^{(1)}_\infty (z)<0$ on the segment connecting $2\pi i/\xi$ and $1$, we also have $\operatorname{Re}\Phi^{(s)}_{M}(z)\leq 0$ on the segment $L$ for $s$ sufficently close to 1 and $M$ sufficiently large. This finishes the construction of the paths.
\end{proof}

\section{Acknowledgements}
We are very grateful to Hitoshi Murakami and Tian Yang for their valuable comments and suggestions to this work. We would also like to thank the referees of the Journal for the careful review which improves our paper a lot. Finally, the first author would like to thank the Mathematics department of the Chinese University of Hong Kong for their continuous support.

\vspace{10pt}
\begin{minipage}[t]{0.5\textwidth}
   Ka Ho Wong\\
   Department of Mathematics,\\
   Texas A\&M University,\\
   College Station, \\
   Texas, United States\\
   daydreamkaho@math.tamu.edu\\
\end{minipage}
\begin{minipage}[t]{0.5\textwidth}
  Thomas Kwok-Keung Au\\
   Department of Mathematics,\\
   The Chinese University of Hong Kong,\\
   Shatin,\\
   Hong Kong\\
   thomasau@math.cuhk.edu.hk
\end{minipage}   

\end{document}